\documentclass{article}
\usepackage{arxiv}
\usepackage{hyperref}

\usepackage{times}

\newcommand{\dnn}{\textup{DNN}}

\usepackage{graphicx}
\usepackage{amssymb}
\usepackage{amsmath}
\usepackage{float}
\usepackage{placeins}
\usepackage{natbib}
\usepackage{smile}

\title{Minimax Optimal Deep Neural Network Classifiers Under
	Smooth Decision Boundary}

\author{%
	{\large Tianyang Hu$^{1}$~~~~Ruiqi Liu$^{2}$~~~~Zuofeng Shang$^{3}$~~~~Guang Cheng$^{4}$}\\[1em]
	$^1$Department of Statistics, Purdue University \\
	$^2$Department of Mathematics and Statistics, Texas Tech University\\
	$^3$Department of  Mathematical Sciences, New Jersey Institute of Technology\\
	$^4$Department of Statistics, University of California, Los Angeles\\
}

\begin{document}
	\maketitle
	
	\begin{abstract}%
		Deep learning has gained huge empirical successes in large-scale classification problems. In contrast, there is a lack of statistical understanding about deep learning methods, particularly in the minimax optimality perspective. For instance, in the classical smooth decision boundary setting, existing deep neural network (DNN) approaches are rate-suboptimal, and it remains elusive how to construct minimax optimal DNN classifiers. Moreover, it is interesting to explore whether DNN classifiers can circumvent the ``curse of dimensionality'' in handling high-dimensional data. The contributions of this paper are two-fold. First, based on a localized margin framework, we discover the source of suboptimality of existing DNN approaches.
		Motivated by this, we propose a new deep learning classifier using a divide-and-conquer technique: DNN classifiers are constructed on each local region and then aggregated to a global one. 
		We further propose a localized version of the classical Tsybakov's noise condition, under which statistical optimality of our new classifier is established. Second, we show that DNN classifiers can adapt to low-dimensional data structures and circumvent the ``curse of dimensionality'' in the sense that the minimax rate only depends on the effective dimension, potentially much smaller than the actual data dimension.
		Numerical experiments are conducted on simulated data to corroborate our theoretical results.
	\end{abstract}


	\section{Introduction}
	Deep learning has achieved many breakthroughs in modern classification tasks, especially for natural images \citep{deng2009imagenet, krizhevsky2012imagenet, he2016deep, nguyen2017iris}.
	In contrast to the huge empirical success, statistical understanding of DNN classifiers is still lacking as to why neural networks perform better than traditional classification methods, particularly for high-dimensional structured data. 
	Most theoretical works revolving around neural network classifiers focus on the generalization perspective, developing error bounds for trained classifiers on unseen data \citep{vapnik1999overview, bousquet2002stability, zhang2016understanding,cao2019generalization}. 
	However, derivation of generalization bounds mostly relies on the complexity of the DNN class, which is often independent of the data distribution.
	Deep learning is not better than traditional methods for every data set and the success of DNN classifiers should not only be contributed to the effectiveness of neural networks, but also task-specific properties such as data structures, noise distributions, etc. 
	As a result, almost all generalization error bounds are vacuous and often don't reflect the actual generalization performance \citep{dziugaite2017computing, jiang2019fantastic}. 
	Without involving the concept of minimax optimality, such bounds may provide little guidance for practical applications.

	There is a rich literature on minimax optimal estimation using DNNs \citep{schmidt2020nonparametric, schmidt2019deep, liu2019optimal, kohler2022estimation, liu2022dive}, under which \emph{task-specific} and \emph{statistical optimal} results can be derived.
	By making specific assumptions on the data, the performance of an estimation method can be sharply characterized by the convergence rate of the estimation error (upper bound).
	In the mean time, a task-specific lower bound on the convergence rate can also be derived, independent of the estimation method. 
	When the convergence rate of a certain method matches the lower bound, statistical optimality is achieved. 
	Furthermore, for high-dimensional structured data, if the convergence rate is ``dimension-free'' in the sense that it doesn't depend on the original data dimension $d$, but on an effective dimension $d^*\ll d$, the ``curse of dimensionality'' can be alleviated. 
	It is interesting to explore whether minimax optimal DNN approaches with dimension-free properties can be established in the classification setting. 
	Existing literature on this front is limited, and most works either treat classification as regression by estimating the conditional class probability instead of the decision boundary \citep{kohler2020rate, kohler2020statistical, bos2021convergence, hu2021understanding, shang2022class1, shang2022class2, shang2022class3} or settle for an upper bound on the misclassification risk \citep{kim2021fast, steinwart2007fast, hamm2020adaptive}. 
	Unlike regression problems where one intends to estimate the unknown regression functions, the goal of classification is to recover the unknown \emph{decision boundary} separating different classes. Properties of the decision boundary are critical and the analysis is more demanding. 
	In literature, the smooth decision boundary assumption and Tsybakov's low noise condition \citep{mammen1999smooth} are commonly adopted, whereas minimax optimal DNN classifiers under these assumptions are nonexistent, not even for simpler classification methods such as support vector machine (SVM) \citep{hamm2020adaptive}.
	{There is a clear gap between DNN classifier's empirical success and its suboptimality.}
	In this work, we aim to address the issue of rate suboptimality in the smooth boundary setting and also provide new insights to how DNNs avoid the ``curse of dimensionality". 
	The main contributions of this paper are: 
	\begin{itemize}
		\item
		By dissecting the 0-1 loss excess risk of DNN classifiers, we identify the potential origin of the suboptimality. 
		Motivated by this, we propose a new deep learning classifier using a divide-and-conquer technique: DNN classifiers are constructed on each local region and then aggregated to a global one. 
		We further propose a novel \emph{localized} version of the classical Tsybakov's noise condition, which enables us to develop both lower and upper bounds on the convergence rate and establish statistical optimality for DNN classifiers with proper architectures, i.e., depth, width, sparsity, etc. 
		
		\item
		In the proposed localized smooth decision boundary setting, we show that DNN classifiers can adapt to low-dimensional structures underlying the high-dimensional data and circumvent the ``curse of dimensionality" in the sense that the optimal rate of convergence only depends on some effective dimension $d^*$, potentially much smaller than the data dimension.
	\end{itemize}

	This work is \emph{the first} to establish the \emph{minimax optimal} convergence rate for DNN classifiers under the classical smooth decision boundary setting, with a novel localized Tsybakov's noise condition. 
	The proposed separation condition facilitates a finer-grained understanding of classification, which, DNN classifiers can take full advantage of by utilizing the divide-and-conquer strategy. 
	The optimality proof relies on delicate constructions of DNNs where the representation power, structural flexibility, composition nature all play critical roles.
	The ability to adapt to low-dimensional data structures while achieving the optimal convergence rate showcases the power of deep learning and provides insights for understanding DNN's empirical success in classifying high-dimensional data.

	\section{Preliminary}
	
	Let $\QQ$ be the Lebesgue measure on $\RR^d$. 
	For any function $f(\bx):\cX\to\RR$, denote  
	$\|f\|_\infty=\sup_{\bx\in\cX}|f(\bx)|$ and $\|f\|_p=(\int_{\cX} |f(\bx)|^p d\bx)^{1/p}$ for $p\in \NN$. 
	For a vector $\bx$, $\|\bx\|_p$ denotes its $p$-norm, for $1\leq p \leq \infty$. 
	$L_p$ and $l_p$ are used to distinguish function norms and vector norms.
	For two given sequences of real numbers $\{a_n\}_{n\in \NN}$ and $\{b_n\}_{n\in \NN}$, we write $a_n\lesssim b_n$ if there exists a constant $C>0$ such that $a_n\le C b_n$ for all sufficiently large $n$. Let $\Omega(\cdot)$ be the counterpart of $O(\cdot)$ that $a_n=\Omega(b_n)$ means $a_n\gtrsim b_n$.
	In addition, we write $a_b\asymp b_n$ if $a_n\lesssim b_n$ and $a_n\gtrsim b_n$.
	For $a,b \in \RR$, denote $a\vee b =\max\{a, b\}$ and $a\wedge b =\min\{a, b\}$.   
	We use $\II$ to denote the indicator function.
	For a set $G\subset \RR^d$, denote $\partial G$ as its boundary and $G^\circ$ as its interior. 
	For two sets $G_1, G_2\subset \RR^d$, let their symmetric difference be $
	d_\triangle(G_1, G_2)= \QQ(G_1\triangle G_2)= \QQ\left((G_1\backslash G_2)\cup (G_2\backslash G_1)\right).$

	\subsection{Binary classification} 
	Denote $\bx\in\cX\subset\RR^d$ as the feature vector and $y\in\{-1, 1\}$ as the label. 
	Assume the classes are balanced, i.e., $\PP(y=1)=1/2$, and 
	\[
	\bx\ |\ y=1\sim p(\bx), \quad \bx\ |\ y=-1\sim q(\bx),
	\] where $p, q$ are two bounded densities on $\cX$. 
	The goal of classification is to find the optimal classifier $C^*$  with the lowest expected \emph{0-1 loss} 
	\[
	R(C) = \EE_{\bx}[\II\{C(\bx)\ne y\}].
	\]
	$C^*$ corresponds to the optimal decision region $G^* := \{\bx\in\cX, p(\bx)- q(\bx)\ge 0\}$ by assigning label 1 if $\bx\in G^*$ and label $-1$ otherwise.
	Equivalently, we can consider real-valued functions $f$ and classification can be done by $C(\bx) = \mbox{sign}(f(\bx))$. 
	With a slight abuse of notation, we use $R(C), R(G)$ and $R(f)$ interchangeably. 
	
	Without knowing the data distribution, $C^*$ is not directly calculable. 
	Instead, we observe data set $\cD=\{(\bx_i,y_i)\}_{i=1}^n$ and estimate $C^*$ by the empirical risk minimizer, i.e., 
	\[
	\hat{C}=\argmin R_n(C), \mbox{ where } R_n(C)=\sum_{i=1}^n \II\{C(\bx_i)\ne y_i\}.
	\]
	In practice, the empirical 0-1 loss minimizer is not computationally feasible because minimizing such loss over $\cC_n$ is NP hard \citep{bartlett2006convexity}. 
	An alternative approach is to replace the 0-1 loss with other computationally friendly surrogate losses $\phi:\RR\to \RR_+$, e.g. hinge loss $\phi(z) = (1-z)_+ = \max\{1-z, 0\}$, logistic loss $\phi(z) = \log(1+\exp(-z))$, etc. 
	Even though surrogate losses are used in practice, 0-1 loss gives the most fundamental characterization of the classification problem and is of great theoretical interest. 
	This work focuses on 0-1 loss and surrogate losses such as hinge loss, cross-entropy are reserved for future work. 
	
	The performance of an estimated classifier $\hat{C}$ can be evaluated by its \emph{excess risk} defined as  
	\[
	\cE(\hat{C}, C^*) = R(\hat{C}) - R(C^*),
	\]
	which measures the closeness between $\hat{C}$ and $C^*$ in terms of the expected 0-1 loss. 
	Convergence rate refers to how fast $\cE(\hat{C}, C^*)$ converges to zero with respect to $n$ and it is the focus on this paper. 
	Knowing the relationship between $C^*$ and $G^*$, classification can also be seen as {nonparametric estimation of sets}, and the 0-1 excess risk of an estimator $\hat{G}$ can be written as
	\begin{equation*}
		\cE(\hat{G}, G^*)=d_{p, q}(\hat{G},G^*) =\int_{\hat{G}\triangle G^*}|p(\bx) - q(\bx)|\QQ(d\bx).
	\end{equation*}
	
	There are two key factors governing the rate of convergence in classification: the complexity of the decision boundary and how concentrated the data are near the decision boundary, which we will introduce in Section \ref{sec:smooth_boundary} and \ref{sec:separation} respectively.

	\subsection{Smooth boundary assumption} \label{sec:smooth_boundary}
	In statistics literature, one of the most common boundary assumptions is called the smooth boundary fragments \citep{mammen1999smooth, tsybakov2004optimal, korostelev2012minimax}. 
	Recall that a function has H\"older smoothness index $\beta$ if all partial derivatives up to order $\lfloor \beta \rfloor$ exist and are bounded and the partial derivatives of order $\lfloor \beta \rfloor$ are $\beta - \lfloor \beta \rfloor$ H\"older continuous, where $\lfloor \beta \rfloor$ denotes the largest integer strictly smaller than $\beta.$ The ball of $\beta$-H\"older functions with radius $R$ is then defined as
	\begin{align}
		\label{eqn:smooth_function}
		\cH(d,\beta,R) = \Big\{ 
		&f:\mathbb{R}^d \rightarrow \mathbb{R} : 
		&\sum_{\balpha : |\balpha| < \beta}\|\partial^{\balpha} f\|_\infty + \sum_{\balpha : |\balpha |= \lfloor \beta \rfloor } \, \sup_{\stackrel{\bx, \by \in D}{\bx \neq \by}}
		\frac{|\partial^{\balpha} f(\bx) - \partial^{\balpha} f(\by)|}{|\bx-\by|_\infty^{\beta-\lfloor \beta \rfloor}} \leq R
		\Big\},
	\end{align}
	where we use multi-index notation, i.e., $\partial^{\balpha}= \partial^{\alpha_1}\ldots \partial^{\alpha_r}$ with $\balpha =(\alpha_1, \ldots, \alpha_r)\in \mathbb{N}^r$ and $|\balpha| :=|\balpha|_1.$ $R$ controls the radius of the function space and as long as $R<\infty$, the function space doesn't fundamentally change. For simplicity, we omit the radius and write $\cH(d,\beta,R)$ as $\cH(d,\beta)$.
	
	For $d\ge 2$, let $\bx_{-d} = (x_1,\cdots,x_{d-1})$.
	The smooth boundary fragment setting assumes the optimal set $G^*$ to have the form
	\begin{align}
		\label{eqn:boundary_fragment}
		G_{f^*} := \{\bx\in\RR^d: f^*(\bx_{-d})-x_d\ge 0, f^*\in\cH(d, \beta)\}.
	\end{align} 
	Denote $\cG^*_\beta = \{G_{f^*}: f^*\in\cH(d, \beta) \}$ to be all such sets (hypothesis class). 
	Due to the smoothness of $f^*$, the complexity of the hypothesis class $\cG^*_\beta$, which is usually measured by \textit{bracketing entropy}, can be controlled. To be more specific, for any $\delta>0$, the bracketing number $\cN_B(\delta, \cG, d_\triangle)$ is the minimal number of set pairs $(U_j, V_j)$ such that 
	\begin{itemize}
		\item [(a)]
		For each $j$, $U_j\subset V_j$ and $d_\triangle(U_j, V_j)\le \delta$;
		\item[(b)]
		For any $G\in \cG$, there exists a pair $(U_j, V_j)$ such that $U_j\subset G\subset V_j$.
	\end{itemize}
	Simply denote $\cN_B(\delta)=\cN_B(\delta, \cG, d_\triangle)$ if no confusion arises. 
	The bracketing entropy is defined as $H_B(\delta)=\log{\cN_B(\delta, \cG, d_\triangle)}$. In the smooth boundary fragment case, \citet{mammen1999smooth} showed that 
	\[
	H_B(\delta,\cG_{\beta}^* ,d_\triangle) \lesssim \delta^{-\frac{d-1}{\beta}}.
	\]
	For more flexibility, the boundary fragments assumption can be extended to be more general that $G^*$ consists of unions and intersections of smooth hyper-surfaces defined as in \eqref{eqn:boundary_fragment} \citep{kim2021fast, tsybakov2004optimal}. 
	In this work, we focus on the basic case \eqref{eqn:boundary_fragment} for simplicity.

	\subsection{Separation assumption} \label{sec:separation}
	The following {Tsybakov's noise condition} \citep{mammen1999smooth} quantifies how separated $p$ and $q$ are near the decision boundary:
	\begin{itemize}
		\item[(N)] There exists constants $c, T>0$ and $\kappa\in[0,\infty]$ such that for any  $0\le t\le T$,
		$$\QQ\left(\{\bx:|p(\bx)-q(\bx)|\le t\}\right)\le c t^\kappa.$$
	\end{itemize}
	(N) is a population-level assumption and $\kappa$ is referred to as the {noise exponent}. (N) trivially holds for $\kappa= 0$. The bigger the $\kappa$, the more separated the data around the decision boundary, and hence, the easier the classification. 
	In the extreme case where $p, q$ have different supports, $\kappa$ can be arbitrarily large ($\infty$) and the boundary recovery is the easiest. To another extreme where $\QQ\{\bx\in\cX: p(\bx)=q(\bx)\}>0$, there exists a region where different classes are indistinguishable. In this case, $\kappa = 0$ and the optimal decision boundary is hard to estimate in that region.

	Under the smooth boundary fragment assumption \eqref{eqn:boundary_fragment} with smoothness $\beta$ and the Tsybakov's noise condition (N) with noise exponent $\kappa$, \cite{mammen1999smooth} showed that the optimal rate of convergence for the 0-1 loss excess risk is 
	\begin{equation}
		\label{eqn:optimal_rate_classification}
		\inf_{C_n\in\cC}\sup_{G^*\in \cG^*_\beta} \cE(C_n, G^*) = \Omega\rbr{n^{-\frac{\beta(\kappa+1)}{\beta(\kappa+2) + (d-1)\kappa}}},
	\end{equation}
	where $\cC$ can be any classifier family. 
	Note that the ``curse of dimensionality" does occur in this bound. As $d$ gets larger, the rate becomes extremely slow.

	\subsection{DNN in classification}
	We consider DNN with rectified linear unit (ReLU) activation that $\sigma(z) = \max\{z,0\}$. For a ReLU neural network $f_n$ indexed by the sample size, let $L_n$ be the number of layers, $N_n$ be the maximum width of all the layers, $S_n$ be the total number of non-zero weights, $B_n$ be maximum absolute values of all the weights.
	Denote $\cF_n = \cF(L_n, N_n, S_n, B_n)$ as a DNN family with structural constraints specified by $L_n, N_n, S_n, B_n$. In this work, depth $L_n$ and width $N_n$ are the primary focuses.

	Convergence rate of DNN classifiers has been investigated in literature. 
	\cite{kim2021fast} derived fast convergence rates of ReLU DNN classifiers learned using the hinge loss ($\phi(z) = \max(0, 1-z)$). 
	Under the smooth boundary fragment assumption \eqref{eqn:boundary_fragment} and Tsybakov's noise condition (N), the empirical hinge loss minimizer
	\[
	\hat{f}_{\phi, n}=\argmin_{f\in \cF_n}\frac{1}{n}\sum_{i=1}^n\phi(y_if(\bx_i)),
	\]
	within some DNN family with carefully selected $L_n,N_n,S_n, B_n$ satisfies 
	\begin{align}
		\label{eq:rate1}
		\sup_{G^*\in \cG^*_{\beta}}
		\EE\left[\cE(\hat{f}_{\phi, n}, G^*)\right]\lesssim  
		\left(\frac{\log^3 n}{n}\right)^{\frac{\beta (\kappa+1)}{\beta (\kappa+2)+(d-1)(\kappa+1)}}.
	\end{align}
	Their result is rate \emph{suboptimal} comparing to the minimax lower bound (\ref{eqn:optimal_rate_classification}). 
	To the best of the authors' knowledge, this is the only existing convergence rate result for DNN classifiers under the smooth boundary setting. 
	The same suboptimal rate is referred to as ``best known rates" for SVM \citep{hamm2020adaptive}. 
	Other theoretical works on convergence rate of DNN classifiers are carried out under other assumptions, e.g., separable \citep{zhang2000convergence, hu2021understanding}, teacher student setting \citep{hu2020sharp}, smooth conditional probability \citep{audibert2007fast, steinwart2007fast, kohler2020rate}, etc. 
	Classification by estimating the conditional probability is usually referred to as "plug-in" classifiers and it's worth noting that it essentially reduces classification to regression. In comparison, estimating the decision boundary is a more fundamental setting for classification \citep{hastie2009elements} and it is the focus of this work.

	\section{Main results -- optimal rate of convergence} 
	\label{sec:rate}
	Under the classical smooth boundary fragment setting, DNN classifiers have not been proved optimal in recovering the decision boundary, which casts doubt on the empirical success of DNNs in classification. 
	In this section, we first investigate the potential reason behind the suboptimality.
	
	The noise exponent $\kappa$ in the Tsybakov's noise condition (N) plays a critical role in determining the convergence rate. In general, bigger $\kappa$ implies better separation, and hence easier boundary recovery. However, if we take a closer look at the excess risk by decomposing it into the stochastic (empirical) error and the approximation error (using DNN to approximate smooth functions), bigger $\kappa$ benefits the former but is harmful for the later (the approximation error is amplified by a power of $1+1/\kappa$ as in Lemma \ref{lemma:d_ineq}). In other words, the bottleneck for the stochastic error is the region with the smallest separation while that for the approximation error is the region with the largest separation. 
	If the separation along the decision boundary is inconsistent with sharp changes, the overall estimation performance of DNN classifiers could be harmed. The optimal rate of convergence can be achieved under extra separation consistency condition. 
	Complementary to the classical Tsybakov's noise condition (N), which indicates that the separation is not too small, the following condition ensures that the separation cannot be too large as well. 
	\begin{itemize}
		\item[(N$^+$)] There exist constants $c_1, T>0$ and $\kappa\in[0,\infty]$ such that for any $0\le t\le T$,
		$$\QQ\left(\{\bx\in G:|p(\bx)-q(\bx)|\le t\}\right)\ge c_1 t^{\kappa}$$
		holds for any positive-measure set $G\subset\cX$ containing the decision boundary, i.e., $\partial G^*\cap G^\circ$ is not empty.
	\end{itemize}
	
	If (N) and (N$^+$) both hold with the same exponent $\kappa$, the separation is guaranteed to be consistent along the decision boundary.  
	The following theorem utilizes this extra condition and shows that optimal rate of convergence can be achieved for DNN classifiers. 
	\begin{lemma}(Informal)
		\label{lemma:thm1}
		Under assumptions (N) and the smooth boundary fragment assumption (\ref{eqn:boundary_fragment}), if we further assume (N$^+$), then the empirical 0-1 loss minimizer within a ReLU DNN family with proper size achieves the optimal 0-1 loss excess risk convergence rate of  ${n^{-\frac{\beta(\kappa+1)}{\beta(\kappa+2) + (d-1)\kappa}}}$.
	\end{lemma}
	
	Lemma \ref{lemma:thm1} is a special case of Theorem \ref{thm:local_estimator} and the proof can be found in the Section \ref{sec:rate_local}. The improvement to optimality stems from a better approximation error bound, which follows from the separation consistency guaranteed by condition (N$^+$). 
	However, assuming identical separation over the whole support, i.e., the same $\kappa$ for both (N) and (N$^+$), is too strong. 
	It becomes much more realistic if we put this assumption on a small neighborhood, or a small segment of the decision boundary. If the densities $p, q$ are not too irregular, the separation is expected to be locally consistent.  
	
	To this end, we conduct convergence analysis in a local region of the decision boundary and consider a localized version of the classical Tsybakov's noise condition. For simplicity, $\cX$ is assumed to be $[0,1]^d$ in the remaining part of the paper. Extension to compact $\cX$ is straightforward.
	
	\subsection{Localized separation condition}
	
	Recall the boundary assumption that $G^*=\{\bx\in[0,1]^d: f^*(\bx_{-d})-x_d \ge 0\}$ where $f^*$ is some smooth function from $\RR^{d-1}$ to $\RR$. Denote the decision boundary to be $\partial G^* := \{\bx\in[0,1]^d: f^*(\bx_{-d})=x_d\}$. 
	Every point $\bx \in \partial G^*$ can be written as $\bx = (\bx_{-d}, f^*(\bx_{-d}))$.  
	Without loss of generality, assume $\QQ(\partial G^*) = 0$. 
	Notice that even though $f^*$ defines the decision boundary, it has nothing to do with the separation. The separation condition, e.g., (N), depends on properties of the data densities $p, q$ around the decision boundary. 
	To quantify the separation, define for each $\bx_{-d}\in[0,1]^{d-1}$,  
	$$m_{\bx_{-d}}(t):=|p((\bx_{-d}, f^*(\bx_{-d})+t)) - q((\bx_{-d}, f^*(\bx_{-d})+t))|,$$
	which captures the how $|p(\bx)-q(\bx)|$ changes along the direction of $x_d$ on each point of the decision boundary. 
	For ease of notation, we write $m_{\bx_{-d}}(t)$ and $m_{\bx}(t)$ when no confusion arises. 
	Notice that $m_{\bx}(0) = 0$ by definition and we want to characterize how $m_{\bx}(t)$ behaves when $t$ is close to 0, i.e., close to the decision boundary.
	Moreover, for every $\bx\in \partial G^*$, define the localized separation exponent
	\begin{equation*} 
		K(\bx) = \sup\{k\ge 0: \lim_{t\to 0} \frac{m_{\bx}(t)}{|t|^{1/k}} >0\}.
	\end{equation*} 
	$K(\bx)$ characterizes the separation condition locally at $\bx_{-d}$, i.e. how separated are $p$ and $q$ on each point of the decision boundary, along the direction of $x_d$. 
	For some $\bx\in \partial G^*$, if the directional derivative of $p-q$ tangent to the decision boundary is finite and non-zero, then $K(\bx) = 1$. 
	Note that if $K(\bx)$ is regular enough, e.g., continuous, Tsybakov's noise condition (N) with exponent $\kappa$ implies that $\kappa\le \inf_{\bx\in \partial G^*} K(\bx)$ and (N$^+$) implies $\kappa\ge \sup_{\bx\in \partial G^*} K(\bx)$. 
	Consider the following separation conditions which form a localized version of the Tsybakov's noise condition. 
	\begin{itemize}
		\item[(M1)] 
		There exists $\epsilon_0>0$ small enough and a constant $0<C_{\epsilon_0}<\infty$ such that for all $\bx\in \partial G^*$ and any $0<t<\epsilon_0$, 
		$$\frac{1}{C_{\epsilon_0}} \le \frac{m_{\bx}(t)}{|t|^{1/K(\bx)}} \le C_{\epsilon_0}.$$
	\end{itemize}
	\begin{itemize}
		\item[(M2)] 
		$K(\bx)$ is $\alpha$-H\"older continuous for some $0<\alpha\le1$, i.e.
		there exists constant $C_K$ such that for any $\bx_1,\bx_2\in \partial G^*$, $$|K(\bx_1)-K(\bx_2)|\le C_K\|\bx_1-\bx_2\|_2^\alpha.$$
	\end{itemize}
	
	\begin{figure}
		\centering
		\includegraphics[scale = 0.5]{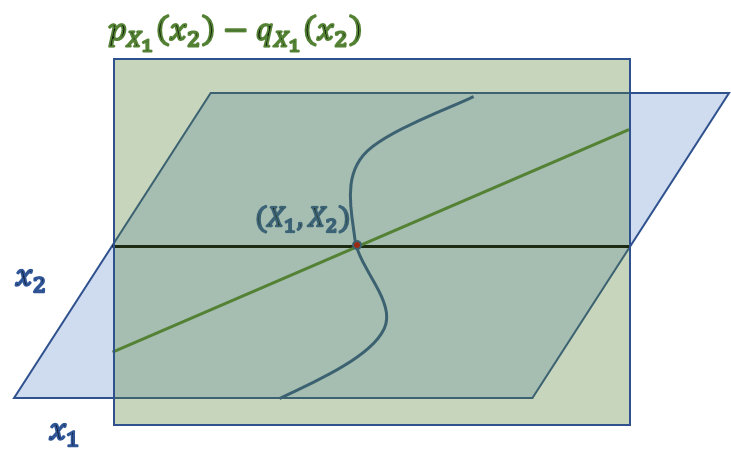}
		\caption{ 
			Illustration of the localized separation condition (M1) when $d=2$. The blue plane represents $\cX$ and the curved blue line is the optimal decision boundary. Fix some $(x_1,x_2)$ in the decision boundary. Along the $x_1$ direction (black line), the density difference $p-q$ is plotted in the green plane as the green line. (M1) defines the noise exponent $K(x_1)$ locally at $x_1$. If $p-q$ is linear as shown, $K(x_1) = 1$.}
		\label{fig:local}
	\end{figure}
	
	\begin{remark}[Justification of (M1,M2)]
		The proposed separation condition is a localized version of the classical Tsybakov's noise condition (N) where (M1) corresponds to (N) in a local region and (M2) ensures consistency among nearby regions.
		(M1,M2) is stronger than (N), only because it’s a finer characterization of (N) and we expect it to hold if the data distribution is not too irregular.
		\textbf{Many existing settings are special examples of (M1,M2)}, e.g., in the separable case, $K(\bx)\equiv \infty$; 
		if $p,q$ are piecewise linear, then (M) holds with $K(\bx)\equiv 1$; 
		if $p, q$ are H\"older smooth with smoothness more than 2, e.g., Gaussian mixtures, then (M1,M2) holds and $K(\bx)\le 1$. 
		Like (N), our condition is more of theoretical interest and hard to verify for real data. Still, the developed results can contribute to the understanding of DNN's empirical success in classification, e.g., how they can take advantage of better class separation and low-dimensional data structures.
	\end{remark}
	(M1,M2) is a slightly stronger characterization of the boundary separation. 
	The following lemma shows that (M1) implies (N) and (N$^+$). The proof can be found in Section \ref{sec:proof}.
	\begin{lemma}
		\label{lemma:m1}
		Denote $\kappa^- = \inf_{\bx\in \partial G^*} K(\bx)$ and $\kappa^+ = \sup_{\bx\in \partial G^*} K(\bx)$. Then condition (M1) implies that (N) holds with $\kappa = \kappa^-$ and (N$^+$) holds with $\kappa=\kappa^+$. 
	\end{lemma}
	Although (M1,M2) and (N) are closely related, the fundamentals of the classification problem, measured by the convergence lower bound, might be different. 
	Below we investigate the 0-1 loss excess risk lower bound under our proposed separation condition. 
	\begin{theorem}
		\label{thm:lower}
		Under the smooth boundary fragments setting (\ref{eqn:boundary_fragment}) with smoothness $\beta$. Assume condition (M1) and let $\kappa^- =\inf_{\bx\in \partial G^*} K(\bx)$. For any function space $\cF$, the 0-1 loss excess risk has the following lower bound,
		\begin{align*}
			\inf_{\hat{f}\in\cF}\sup_{G^*\in\cG^*_\beta}\EE[\cE(\hat{f},G^*)]\gtrsim  
			\left(\frac{1}{n}\right)^{\frac{\beta(\kappa^-+1)}{\beta(\kappa^-+2) + (d-1)\kappa^-}}.
		\end{align*}
	\end{theorem}
	The convergence lower bound under (M1) only depends on $\kappa^-$ and has the same form compared to (\ref{eqn:optimal_rate_classification}) under (N). This is expected since (M1,M2) is a localized version of (N).

	\subsection{Localized convergence analysis}
	Defining function $K(\bx)$ enables us to consider local convergence behaviours. As can be seen from Lemma \ref{lemma:thm1} and Lemma \ref{lemma:m1}, if the separation described by $K(\bx)$ is consistent, i.e., $\kappa^+\approx\kappa^-$, the convergence rate can be improved to be optimal. Such consistency is expected within any small-enough local region due to condition (M2). 
	To further investigate the interplay between $\kappa^+$ and $\kappa^-$, we conduct localized convergence analysis in this section.

	There are many ways to construct local regions. Since we are considering the smooth boundary fragment setting (\ref{eqn:boundary_fragment}), where $x_d$ is a special dimension, a natural choice for the local regions is equal-sized grids in the $\bx_{-d}$ space. 
	Choose integer $M>0$ and divide $[0,1]^d$ into $M^{d-1}$ regions 
	\[
	[0,1]^d = \bigcup_{j_1,\ldots,j_{d-1}=1}^M D_{(j_1,\ldots,j_{d-1})},
	\] 
	where $D_{(j_1,\ldots,j_{d-1})} := \{\bx\in[0,1]^d: x_1\in[\frac{j_1-1}{M},\frac{j_1}{M}),\cdots,x_{d-1}\in[\frac{j_{d-1}-1}{M},\frac{j_{d-1}}{M}]\}$. Denote the grid points to be $\bar{\bx}_{j_1,\ldots,j_{d-1}}$. 
	For ease of notation, let $\bj_{-d}=(j_1,\cdots,j_{d-1})$ and denote $J_M$ as all $M^{d-1}$ combinations of ${\bj_{-d}}$'s described above. 
	Correspondingly, divide the data set as $\cD = \cup_{\bj_{-d}\in J_M}\cD_{\bj_{-d}}$ where $\cD_{\bj_{-d}}=\{(\bx,y)\in\cD: \bx\in D_{\bj_{-d}}\}$.
	Similarly, the 0-1 loss can be decomposed into 
	\begin{align*}
		d_{p,q}(\hat{G}_n,G^*) =& \int_{\hat{G}_n\triangle G^*} |p(\bx)-q(\bx)|d\bx = \sum_{\bj_{-d}\in J_M} \int_{(\hat{G}_n\triangle G^*)\cap D_{\bj_{-d}}} |p(\bx)-q(\bx)|d\bx\\
		:=& \sum_{\bj_{-d}\in J_M} d_{\bj_{-d}}(\hat{G}_n,G^*).
	\end{align*}
	The empirical 0-1 loss, $R_n(f) = \frac{1}{n}\sum_{i=1}^n\II{\{f(\bx_i) y_i<0\}}$, can also be decomposed into $M^{d-1}$ parts, i.e., $R_n(G) = \sum_{\bj_{-d}\in J_M}R_{n,\bj_{-d}}$ where 
	$$ R_{n,\bj_{-d}}=\frac{1}{|\cD_{\bj_{-d}}|}\sum_{i=1}^n\II{\{\bx_i\in D_{\bj_{-d}}: f(\bx_i) y_i<0\}}.$$

	Now we focus on each local region $D_{\bj_{-d}}$.
	Similar to that in the whole region $[0,1]^d$, we have the following theorem on the local 0-1 loss excess risk convergence rate. 
	\begin{theorem}
		\label{thm:local_estimator}
		Under assumption (M1), further assume that for some $\bj_{-d} \in J_M$, $\kappa^-\le K(\bx)\le\kappa^+$ for all $\bx\in D_{\bj_{-d}}$.
		Let $\tilde{\cF}_n$ be a ReLU DNN family\footnote{$\tilde{\cF}_n$ is used to denote DNN family for local estimation. ${\cF}_n$ is reserved for the global estimator DNN family.} with size in the order of 
		\[
		\tilde{N}_n \tilde{L}_n \asymp {n^{\ \frac{\kappa^+(\kappa^-+1)(d-1)/2}{(\kappa^-+2)(\kappa^+ +1)\beta + (d-1)\kappa^+(\kappa^-+1)}}}\cdot \log^2 (n).
		\]
		Let the empirical 0-1 loss minimizer be
		\begin{equation}\label{regular:class}
			\hat{f}_{n, {\bj_{-d}}}:=\argmin_{f\in\tilde{\cF}_n}R_{n,{\bj_{-d}}}(f).
		\end{equation}
		Then the 0-1 loss excess risk satisfies
		\begin{align*}
			\sup_{G^*\in \cG^*_\beta }\EE(R_{{\bj_{-d}}}(\hat{f}_{n, {\bj_{-d}}})-R_{{\bj_{-d}}}({G^*})) = \tilde{O} \rbr{n^{-\frac{(\kappa^-+1)\beta}{(\kappa^-+2)\beta + \mathbf{\rbr{\frac{\kappa^-+1}{\kappa^+ +1}}}(d-1)\kappa^+}}},
		\end{align*}
		where $\tilde{O}(\cdot)$ hides the $\log(n)$ terms. 
	\end{theorem}
	Note that (\ref{regular:class}) with $M=1$ corresponds to regular DNN classifiers considered by \cite{kim2021fast} and \cite{hamm2020adaptive}. Our result is sharper in the sense that $\kappa^+$ appears in both the convergence rate and the classifier size. 
	The requirement for the network size comes from the DNN approximation literature can can be quite flexible. In Theorem \ref{thm:local_estimator}, no constraint is put on $B_n$ and $S_n$ since we are using the results from \cite{lu2020deep}.
	
	\begin{remark}[Origin of sub-optimality]
		\label{rmk1}
		Our local convergence rate in Theorem \ref{thm:local_estimator} closely resembles those established under the original Tsybakov's noise condition (N). 
		On one hand, the convergence rate bottleneck is indeed the minimum value of $K(\bx)$ in that region and $\kappa^-$ plays the same role as $\kappa$ in (N).
		On the other hand, the extra term in the denominator $(\kappa^-+1)/(\kappa^+ +1)$ reveals the potential \textit{source of the suboptimality} of existing results. If no assumption is made on $\kappa^+$, one allows $\kappa^+\to\infty$ so that $\kappa^+/(\kappa^++1)\to 1$, the convergence rate reduces to the suboptimal one (\ref{eq:rate1}) obtained in \cite{kim2021fast} and \cite{hamm2020adaptive}. However, if $\kappa^+=\kappa^-$, i.e., $K(\bx)$ is constant, optimal rate is attainable.  
		This is consistent with our numerical findings in which the regular DNN classifier performs well when $K(\bx)$ is constant; see Figure \ref{fig:acc}. 
		Condition (M2) says that $K(\bx)$ is locally smooth, i.e., $\kappa^+\approx\kappa^-$ as long as the local region is small.
	\end{remark} 
	
	In order to achieve the fastest convergence rate, Theorem \ref{thm:local_estimator} requires $\tilde{\cF}_n$ to have proper size. If $\kappa^-=\kappa^+=\kappa$, the size constraint is 
	\[
	\tilde{N}_n \tilde{L}_n=\tilde{O}\rbr{n^{\ \frac{\kappa(d-1)/2}{(\kappa+2)\beta + (d-1)\kappa}}}.
	\]
	The bigger the $\kappa$, the larger the required size, and the faster the convergence rate. 
	In case the size constraint is not met, the the following corollary gives the corresponding convergence results. 
	
	\begin{corollary}\label{cor}
		Under the same setting as Theorem \ref{thm:local_estimator}, denote $r_0=\frac{\kappa^+(\kappa^-+1)(d-1)}{(\kappa^-+2)(\kappa^+ +1)\beta + (d-1)\kappa^+(\kappa^-+1)}$ and let $\tilde{\cF}_n$'s size satisfy
		$\tilde{N}_n \tilde{L}_n=\Omega(n^{r/2})$ for some $0<r<1$. 
		If $r<r_0$, the approximation error dominates and the 0-1 loss excess risk convergence rate becomes 
		$\tilde{O}\rbr{n^{-\frac{r\beta(\kappa^++1)}{(d-1)\kappa^+}}}$. If $r>r_0$, the stochastic error dominates and the convergence rate is
		$\tilde{O}\rbr{n^{- \frac{(1-r)(\kappa^-+1)}{\kappa^-+2}}}$. 
	\end{corollary}

	\begin{figure}[t]
		\centering
		\includegraphics[width=0.45\textwidth]{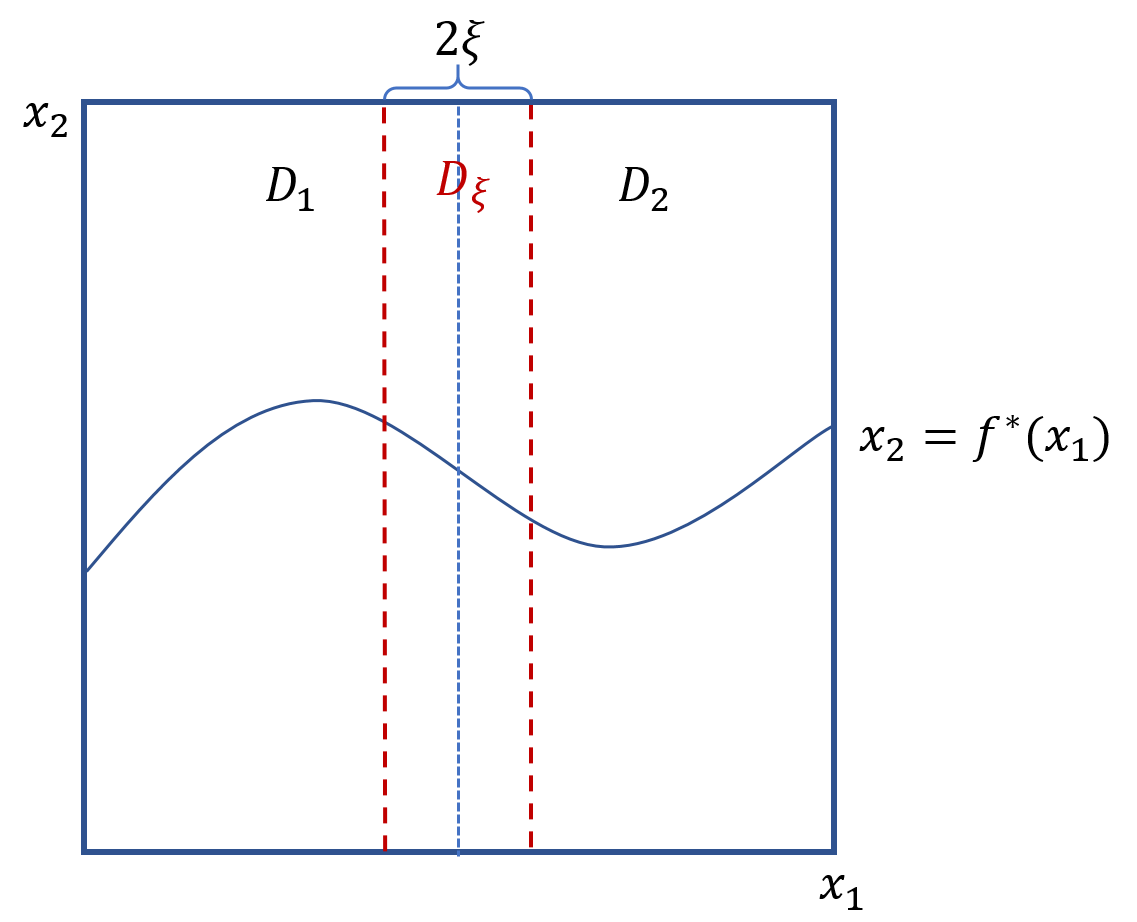}
		\caption{Illustration of region $D_\xi$ in $d=2, M=1$ case.}
		\label{fig:E}
	\end{figure}

	\subsection{Construction of the global estimator}
	\label{sec:construction}
	In this section, we proceed from a localized analysis to the global one and evaluate the overall excess risk convergence rate.
	The goal is to construct a global classifier that takes advantage of the developed results for each region $D_{\bj_{-d}}$, such that Theorem \ref{thm:local_estimator} can be applied locally. 
	To be more specific, denote the properly sized network family (according to Theorem \ref{thm:local_estimator}) for each region as $\tilde{\cF}_{n,\bj_{-d}}$, with width $\tilde{N}_{n,\bj_{-d}}$ and depth $\tilde{L}_{n,\bj_{-d}}$. 
	Let $\cF_n$ be the DNN family for global estimation, with depth $L_n$ and width $N_n$. Denote the global empirical minimizer within $\cF_n$ to be 
	\begin{equation}
		\label{eqn:f_hat}
		\hat{f}_n:=\argmin_{f\in\cF_n} R_n(f).
	\end{equation}
	Then, it's ideal if for any $\bj_{-d}\in J_M$, $\hat{f}_n$ satisfies
	\begin{equation}
		\label{eqn:global}
		R_{n,\bj_{-d}}({\hat{f}_n}) 
		= \min_{f\in\tilde{\cF}_{n,\bj_{-d}}}R_{n,\bj_{-d}}({f}). 
	\end{equation}
	Note that $\hat{f}_n$ degenerates to the regular DNN classifier when $M=1$. 
	For general $M>1$, we aim to construct ${\cF}_n$ in a divide-and-conquer fashion such that \eqref{eqn:global} is satisfied with high probability. 
	Below we present the construction of $\cF_n$ and the probabilistic argument.

	Recall the definition of $D_{\bj_{-d}}$, where we divide $[0,1]^d$ along $\bx_{-d}$ into $M^{d-1}$ equally sized regions. 
	For some $0<\xi \ll 1/M$, define a small region around all grid points $\bar{\bx}_{\bj_{-d}}$'s as 
	\begin{align*}
		D_\xi = \{\bx\in [0,1]^d: \|\bx_{-d} - \bar{\bx}_{\bj_{-d}}\|_\infty\le \xi, \ \bj_{-d}\in J_M\}.
	\end{align*}
	$D_\xi$ with $d=2, M=2$ is illustrated in Figure \ref{fig:E}. We aim to show that \eqref{eqn:global} holds outside $D_\xi$ for some carefully chosen $\xi$ and $\cF_n$. 
	To this end, define event 
	$$E_\xi:=\{\bx_i\notin D_\xi:\forall i = 1,2,\ldots,n\}.$$
	Since $p(\bx)$ and $q(\bx)$ are both bounded densities (by $c_0$), we have
	\begin{align*}
		\PP(x\in D_\xi) &\le c_0 \QQ(D_\xi)\le 2c_0 M\xi(d-1).
	\end{align*}
	Therefore, if we choose $M$ such that $n M\xi(d-1)\to 0$ as $n\to\infty$, then
	\begin{align*}
		\PP(E_\xi) &\ge (1- 2c_0 M\xi(d-1))^n \to 1.
	\end{align*}
	In the remaining of the analysis, we assume that $E_\xi$ happens. 
	For any $f_{n,\bj_{-d}}\in\tilde{\cF}_{n,\bj_{-d}}$, we make modifications and further construct ${f}^+_{n,\bj_{-d}}$ that satisfies the following properties:
	\begin{enumerate}
		\item[(P1)]\label{p1}
		On $D_{\bj_{-d}}\backslash D_\xi$, ${f}^+_{n,\bj_{-d}}={f}_{n,\bj_{-d}}$;
		
		\item[(P2)]\label{p2}
		Outside $D_{\bj_{-d}}$, ${f}^+_{n,\bj_{-d}}=0$;
		
		\item[(P3)] \label{p3}
		${f}^+_{n,\bj_{-d}}\in\tilde{\cF}^+_{n,\bj_{-d}}$ where $\tilde{\cF}^+_{n,\bj_{-d}}$ is slightly larger than $\tilde{\cF}_{n,\bj_{-d}}$ with depth $\tilde{L}^+_{n,\bj_{-d}} = \tilde{L}_{n,\bj_{-d}} + O(1)$ and width $\tilde{N}^+_{n,\bj_{-d}} = 2\tilde{N}_{n,\bj_{-d}}$. 
	\end{enumerate}
	The construction details and verification of (P1) to (P3) are deferred to Section \ref{sec:p13}.
	Let's proceed with the properties of ${f}^+_{n,\bj_{-d}}$ and $\tilde{\cF}^+_{n,\bj_{-d}}$.
	By (P2), ${f}_{n,\bj_{-d}}^{+}$ is zero outside $D_{\bj_{-d}}$ and we can combine them together to define
	\begin{align}\label{eqn:sum}
		{f}_{n, \Sigma}(\bx_{-d}) = \sum_{\bj_{-d}\in J_M} {f}_{n,\bj_{-d}}^{+}(\bx_{-d}).
	\end{align}
	Easy to see that ${f}_{n, \Sigma}(\bx)$ is still a ReLU network. 
	Correspondingly, define such structured DNN family to be ${\cF}_n$, which is $\tilde{\cF}_{n,\bj_{-d}}^+$ stacked in parallel for $\bj_{-d}\in J_M$. 
	See Figure \ref{fig:size} for illustration. 
	Overall, the size of $\cF_n$ satisfies $L_n\lesssim \max_{\bj_{-d}}\{\tilde{L}_{n, \bj_{-d}}\}$, $ N_n\lesssim M^{d-1}\cdot \max_{\bj_{-d}}\{\tilde{N}_{n, \bj_{-d}}\}.$ 
	Recall that a larger $\kappa$ demands a larger network size for approximation. 
	Thus the requirement on $L_n, N_n$ is mainly determined by the region with the biggest $K(\bx)$. 

	Due to the formulation of $\cF_n$, $\hat{f}_n$ can be written in the form of (\ref{eqn:sum}) as $\hat{f}_n=\sum_{\bj_{-d}\in J_M}\hat{f}_{n, \bj_{-d}}$. 
	Denote $\tilde{f}_{n,\bj_{-d}}$ as the projection of $f^*$ to $\cF_{n,\bj_{-d}}$, i.e., $\tilde{f}_{n,\bj_{-d}} = \argmin_{f\in\cF_{n,\bj_{-d}}}\|f-f^*\|_\infty$.
	Under event $E_\xi$, we have that for any $\bj_{-d}\in J_M$,
	\begin{align}
		\label{eqn:local}
		R_{n,\bj_{-d}}({\hat{f}_n}) 
		= R_{n,\bj_{-d}}(\hat{f}_{n,\bj_{-d}}) 
		=\min_{f\in\tilde{\cF}_{n,\bj_{-d}}}R_{n,\bj_{-d}}({f})
		\le R_{n,\bj_{-d}}(\tilde{f}_{n,\bj_{-d}}).
	\end{align}
	The second equality is guaranteed by event $E_\xi$ and property (P1).
	The last inequality is due to empirical risk minimization and the fact that $\tilde{f}_{n,\bj_{-d}}\in \tilde{\cF}_{n,\bj_{-d}}$. 
	Equation \eqref{eqn:local} indicates that the global empirical minimizer within $\cF_n$ also gives rise to the empirical minimizer locally within each $D_{\bj_{-d}}$.
	
	\begin{figure}[t]
		\centering
		\includegraphics[width=0.55\textwidth]{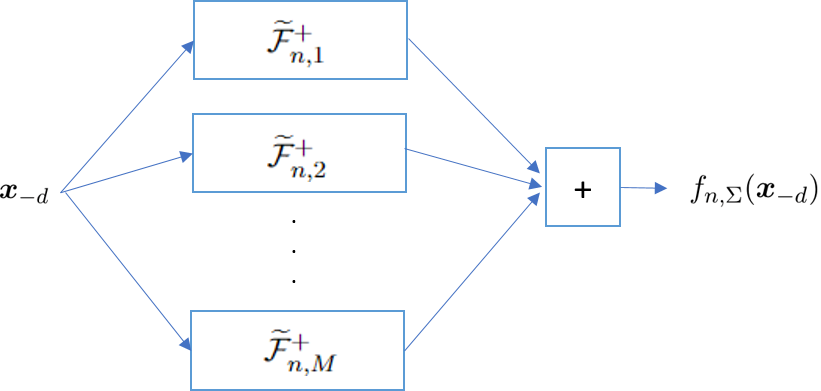}
		\caption{Illustration of the estimator DNN family $\cF_n$ with $d=2$.}
		\label{fig:size}
	\end{figure}

	\begin{remark}[Structured DNN]
		The constructed DNN classifier $\hat{f}_n\in\cF_n$ in Theorem \ref{thm:global_estimator} has special structures and is sparsely connected as illustrated in Figure \ref{fig:size}. 
		Such a structural requirement is not uncommon in nonparametric studies of deep learning where almost all DNN estimators are constructed with special structures \citep{schmidt2020nonparametric, bauer2019deep, imaizumi2018deep}. 
	\end{remark}

	\subsection{Global convergence analysis}
	In this section, we evaluate the convergence rate of the global classifier $\hat{f}_n\in\cF_n$.  
	Through our construction process, we have shown that for any $\bj_{-d}\in J_M$, \eqref{eqn:global} satisfies under $E_\xi$. 
	Intuitively, even though $K(\bx)$ can vary along the decision boundary, the convergence rate will be dominated by the region around $\argmin_{\bx} K(\bx)$. We have the following theorem on the overall convergence rate.

	\begin{theorem}
		\label{thm:global_estimator}
		Under the smooth boundary fragments setting (\ref{eqn:boundary_fragment}),
		assume conditions (M1,M2). Denote $\kappa^- = \inf_{\bx\in[0,1]^d}K(\bx), \kappa^+ = \sup_{\bx\in[0,1]^d}K(\bx)$. 
		Let $\cF_n$ be a ReLU DNN family with proper architectures specified in Section \ref{sec:construction} and size constraint
		\[
		{N}_n {L}_n = O\rbr{n^{\ \frac{\kappa^+(d-1)/2}{(\kappa^++2)\beta + (d-1)\kappa^+}}\cdot \log^{d+1} (n)}.
		\]
		Then, with probability tending to one, the empirical 0-1 loss minimizer within $\cF_n$ satisfies 
		\begin{align*}
			\inf_{\hat{f}_n\in\cF_n}\sup_{G^*\in \cG^*_\beta }\EE(R(\hat{f}_n)-R({G^*})) = \tilde{O}\rbr{n^{-\frac{(\kappa^-+1)\beta}{(\kappa^-+2)\beta+ (d-1)\kappa^-}}}.
		\end{align*}
	\end{theorem}
	Notice that the size constraint on $N_nL_n$ only concerns $\kappa^+$ while the final convergence rate only depends on $\kappa^-$. This is consistent with Remark \ref{rmk1} where we state that the bottleneck for convergence is $\kappa^-$ while that for approximation is $\kappa^+$. The ``with high probability" argument comes from the $E_\xi$, which is an artifact of the proof and may be relaxed. 
	
	The convergence rate in Theorem \ref{thm:global_estimator} matches the lower bound in Theorem \ref{thm:lower} and is statistically optimal up to a logarithmic term. 
	Combining Theorems \ref{thm:local_estimator} and \ref{thm:global_estimator},
	we conclude that, when $K(\bx)$ is highly non-constant in the sense that $\kappa^+\to\infty$, the proposed localized classifier $\hat{f}_n$ has faster rate of convergence than the regular DNN classifier.

	\section{Main results -- breaking the ``curse of dimensionality"}
	\label{sec:dim}
	High-dimensional data are often structured. Taking images as an example, nearby pixels are often highly correlated while distant ones are more independent. The support of $d$-dimensional real images is speculated to be degenerate since not every combination of pixel values is relatable to real-life objects. To understand why deep learning models are not struggling in ultra-high dimensions, it is of great importance to investigate how well they adapt to various structures underlying the data. 
	To this end, many low-dimensional structures have been investigated, e.g., low-dimensional manifold \citep{schmidt2019deep,hamm2020adaptive}, hierarchical interaction max-pooling model \citep{kohler2020rate}, teacher student setting \citep{hu2020sharp}, etc.
	
	In this section, as a proof of concept, we focus on a particular structural assumption -- compositional smoothness structure \citep{schmidt2020nonparametric} -- and show that DNN classifiers can adapt to this low-dimensional structure and optimal convergence rate that only depends on the effective dimension is achievable.

	\subsection{Smooth boundary with compositional structure}
	Recall the smooth boundary fragment setting (\ref{eqn:boundary_fragment}). Instead of $\beta$-smooth, we assume $f^*$ to be compositions of smooth functions. 
	This assumption is first considered in \cite{schmidt2020nonparametric} for the regression function and here we are borrowing it for classification. 
	Assume $f^*$ is of the form 
	\begin{equation}
		\label{comp}
		f^* = h_l \circ h_{q-1} \circ \ldots \circ h_1 \circ h_0
	\end{equation}
	where each $h_i :[a_i,b_i]^{d_i} \rightarrow [a_{i+1},b_{i+1}]^{d_{i+1}}$ and denote the components of 
	$h_i$ by $\{h_{ij}\}_{j=1}^{d_{i+1}}$. Let $t_i$ be the maximal number of variables $h_{ij}$'s depend on. Thus, each $h_{ij}$ is a $t_i$-variate function. 
	We further assume that each function $h_{ij}$ shares the same H\"older smoothness $\beta_i.$ 
	Define
	\[
	\bar{\beta}_i :=\beta_i\prod_{j=i+1}^l (\beta_j\wedge 1),\quad \quad  i^* =  \argmax_{i=0,1, \cdots, l}\ n^{-\frac{2\bar{\beta}_i}{2\bar{\beta}_i+t_i}}.
	\]
	The overall effective smoothness and effective dimension of $f^*$ in \eqref{comp} can be described by $\beta^{*} = \bar{\beta}_{i^*}$ and $d^*=t_{i^*}$, respectively. 
	In this setting, $n^{-{2\beta^{*}}/({2\beta^{*}+d^*})}$ is proven to be the best possible square loss estimation rate from regression \citep{schmidt2020nonparametric}. 
	Let $\cC(d^*, \beta^{*})$ be the corresponding classifier class.
	The compositional smoothness is a generalization of the regular smoothness. If $f^*(\bx)$ is a $(d-1)$-dimensional $\beta$-smooth function as in (\ref{eqn:boundary_fragment}), then $l=0, \beta^{*}=\beta, d^*=d-1$.

	\subsection{Optimal rate of convergence}
	We first establish the following convergence rate lower bound under the proposed setting.  
	\begin{theorem}
		\label{lower1}
		Assume condition (M1) with noise exponent $\kappa =\inf_{\bx\in \partial G^*} K(\bx)$ and the compositional smoothness structure (\ref{comp}) on the boundary fragment assumption (\ref{eqn:boundary_fragment}). The 0-1 loss excess risk has the following lower bound
		\begin{align*}
			\inf_{\hat{f}_n}\sup_{C^*\in\cC(d^*, \beta^{*})}\EE[\cE(\hat{f}_n,C^*)]\gtrsim  
			\left(\frac{1}{n}\right)^{\frac{\beta^{*}(\kappa+1)}{\beta^{*}(\kappa+2) + d^*\kappa}}. 
		\end{align*}
	\end{theorem}
	The convergence rate lower bound adapts to the compositional smoothness structure and only depends on $\beta^{*}$ and $d^*$.
	Next, we evaluate how well DNNs can recover the decision boundary. 
	
	To handle the compositional smoothness, we utilize the corresponding ReLU DNN family in \cite{schmidt2020nonparametric} as our building block and go through the same construction process as in Section \ref{sec:construction}. 
	To distinguish from the regular case in Section \ref{sec:rate}, we add $^*$ to the notations of the network family. 
	Locally, similar to the requirement in Theorem \ref{thm:local_estimator}, we modify $\tilde{\cF}_{n,\bj_{-d}}$ according to \cite{schmidt2020nonparametric} to $\tilde{\cF}_{n,\bj_{-d}}^*$ with $\tilde{L}_n^*\asymp \log(n)$, $$\tilde{N}_n^* \asymp n^{\ \frac{\kappa^+(\kappa^-+1)d^*}{(\kappa^-+2)(\kappa^+ +1)\beta^{*} + d^*\kappa^+(\kappa^-+1)}},\quad \tilde{S}_n^*\asymp n^{\ \frac{\kappa^+(\kappa^-+1)d^*}{(\kappa^-+2)(\kappa^+ +1)\beta^{*} + d^*\kappa^+(\kappa^-+1)}}\cdot \log(n),$$ which is sparsely connected. 
	Similarly, we extend $\tilde{\cF}_{n,\bj_{-d}}^*$ to $\tilde{\cF}_{n,\bj_{-d}}^{+*}$ and stack them together to get $\cF_n^*$ for global estimation as in Section \ref{sec:construction}. 
	Similar to Theorem \ref{thm:global_estimator}, the DNN family $\cF_n^*$ in Theorem \ref{thm:global_estimator_eff} is also constructed with structures and sparsely-connected. The size requirement is slightly different comparing to that in Theorem \ref{thm:global_estimator}, due to a different approximation scheme.

	\begin{theorem}
		\label{thm:global_estimator_eff}
		Under the compositional smoothness setting (\ref{comp}), assume condition (M1,M2) and denote $\kappa^- = \inf_{\bx\in[0,1]^d}K(\bx), \kappa^+ = \sup_{\bx\in[0,1]^d}K(\bx)$. 
		Let $\cF_n^*$ be a ReLU DNN family with proper architectures and size constraint $L_n^*\asymp \log (n)$,
		\[
		{N}_n^*\asymp n^{\ \frac{\kappa^+ d^*}{(\kappa^++2)\beta^{*} + d^*\kappa^+}}\log^{d-1}(n),\quad S_n^*\asymp n^{\ \frac{\kappa^+ d^*}{(\kappa^++2)\beta^{*} + d^*\kappa^+}}\log^{d}(n).
		\]
		Then, with probability $\xrightarrow{n\to\infty} 1$, the empirical 0-1 loss minimizer within $\cF_n$ satisfies 
		\begin{align*}
			\inf_{\hat{f}_n\in\cF_n^*}\sup_{C^*\in \cC(d^*,\beta^{*}) }\EE(R(\hat{f}_n)-R({C^*})) = \tilde{O}\rbr{n^{-\frac{(\kappa^-+1)\beta^{*}}{(\kappa^-+2)\beta^{*}+\kappa^- d^*}}}.
		\end{align*}
	\end{theorem}

	To further illustrate the power of our Theorem \ref{thm:global_estimator_eff}, we consider a special case where the $f^*(\bx)$ is a $(d-1)$-dimensional additive function, i.e., 
	\begin{align}
		\label{additive}
		f^*(\bx_{-d}) = \sum_{i< d}f^*_i(x_i)  = h_1\circ h_0,
	\end{align}
	where $h_0(x_1,\cdots, x_{d-1}) = (g_1(x_1), \cdots, g_{d-1}(x_{d-1}))$ and $h_1(x_1,\cdots, x_{d-1}) = x_1+\cdots+x_{d-1}.$
	In this case, $l=1, \bd= (d-1,d-1), \bt=(1,d-1)$. Assume each $g_i(\bx)$ has the same smoothness $\beta$. Then $\bbeta=(\beta, \infty)$ and we have the following corollary for the convergence rate.
	
	\begin{corollary}
		Under the same setting as in Theorem \ref{thm:global_estimator_eff}, further assume the additive structure (\ref{additive}) on the boundary function. Then the excess risk convergence rate becomes $\tilde{O}\rbr{n^{-\frac{(\kappa+1)\beta}{(\kappa+2)\beta+\kappa}}}$.
	\end{corollary}

	The compositional smoothness assumption (\ref{comp}) may not be the best choice for describing natural data. But as a proof of concept, we have shown that DNN classifiers can adapt to this specific low-dimensional assumption while achieving the optimal convergence rate and avoid the ``curse of dimensionality". 
	In our localized analysis, various other low-dimensional structural assumptions on the decision boundary can be utilized to break the ``curse of dimensionality", e.g., low-dimensional manifold support \citep{schmidt2019deep},  low intrinsic dimension \citep{hamm2020adaptive}, etc. 
	We conjecture that as long as DNNs can adapt to these assumptions in the regression setting and achieve optimal convergence rates, such adaptions can be potentially extended to the smooth boundary fragment classification setting.

	\begin{figure}
		\centering
		\hspace{-2cm}
		\subfigure{
			\label{fig:data}
			\includegraphics[width=0.5\textwidth]{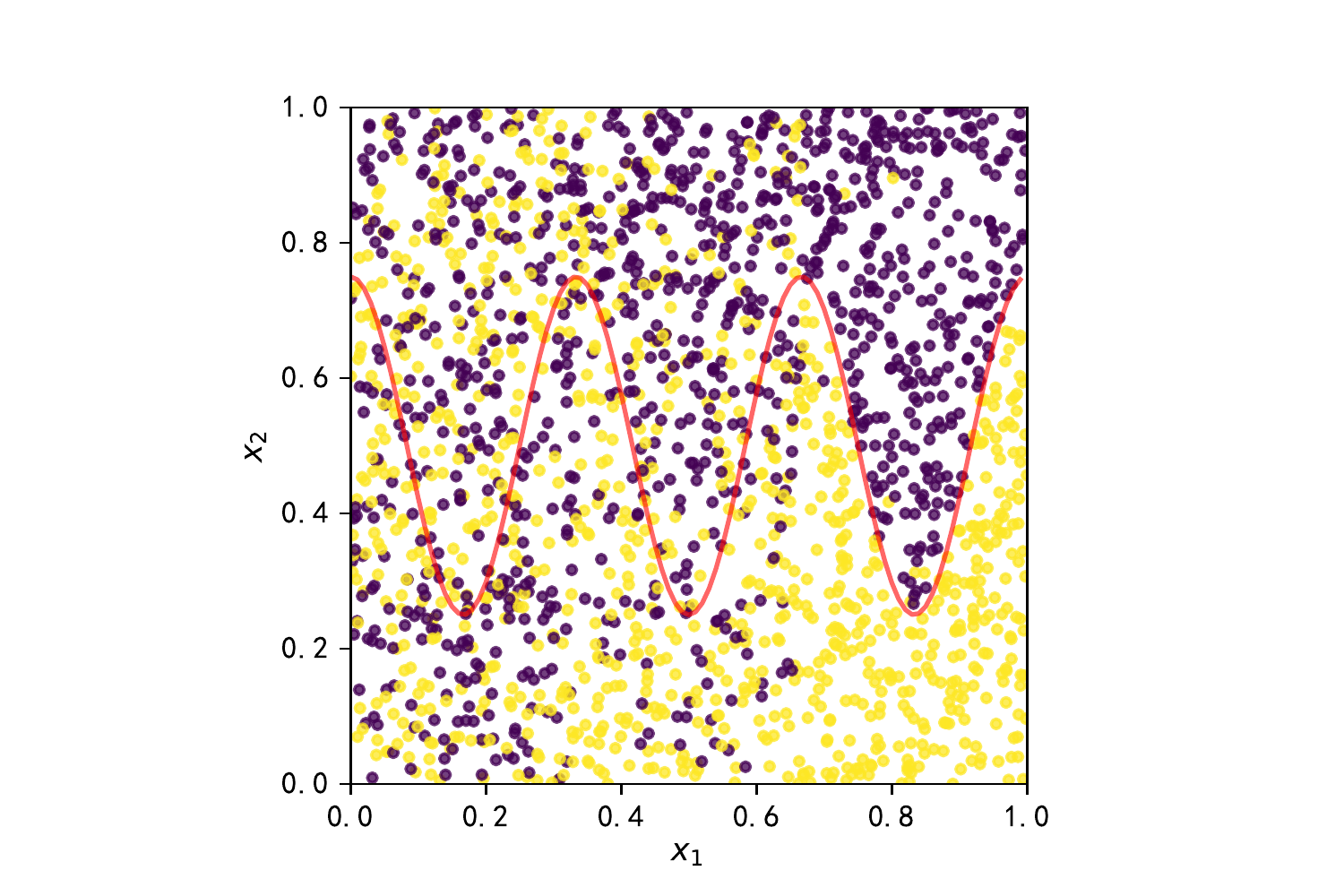}
		}
		\subfigure{
			\label{fig:kappa}
			\includegraphics[width=0.5\textwidth]{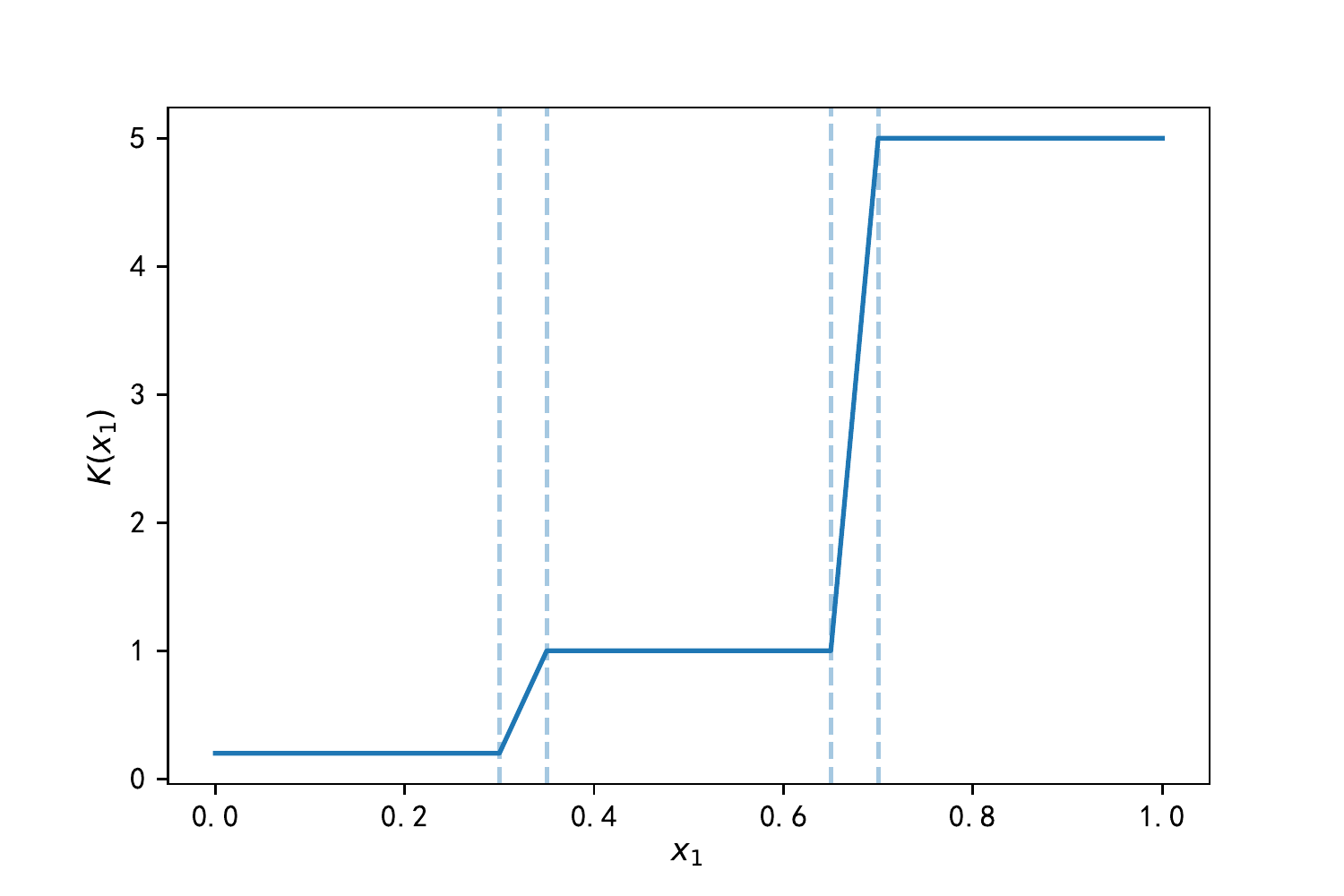}}
		\hspace{-1cm}
		\caption{ (a) 2000 data points in $[0,1]^2$ with $k=100$. The colors represent different classes and the red curve is the optimal decision boundary. From left to right, the separation measured by $K(\bx)$ increases. 
			(b) Illustration of the designed $K(\bx)$ in the synthetic data when $k=5$. }
	\end{figure}
	
	\section{Numerical experiments}
	In this section, we corroborate our theory with numerical experiments on 2-dimensional synthetic data.
	Let the marginal distribution of $\bx$ be uniform on $[0, 1]^2$, i.e., $p(\bx) + q(\bx) \equiv 2$. 
	Consider the optimal decision boundary given by $x_2 = f^*(x_1),$ where $f^*(x_1) =  \cos(6\pi x_1)/4 + 1/2$. Denote $\delta(\bx) = \frac{4}{3}\rbr{x_2 - f^*(x_1)}$, which ranges from $-1$ to $1$, representing the signed distance to the decision boundary along $x_2$.
	Specify the conditional probability as
	\begin{align*}
		2\eta(\bx)- 1 = 
		\sign(\delta(\bx))\cdot {\delta(\bx)}^{K(\bx)}.
	\end{align*}
	It's easy to verify that $\eta(\bx)$ is well-defined and (M1) holds with $K(\bx)$. Due to the symmetry of $f^*$ around $1/2$, we have $\PP(\eta(\bx)>1/2) = 1/2$. 
	Notice that given $\eta(\bx)$, we can write $p(\bx) = 2\eta(\bx)$ and $q(\bx) = 2-2\eta(\bx)$. 
	To reflect the inconsistent separation, we choose $K(\bx)$ to be piecewise linear, where (M2) also holds with $\alpha=1$. Specifically, let $K(x_1)=1/k$ for $x_1\in[0, 0.3]$, $K(x_1)=1$ for $x_1\in[0.35, 0.65]$, and $K(x_1)=k$ for $x_1\in[0.7, 1]$. When $k=1$, (N) and (N$^+$) both hold with $\kappa=1$ and the separation is consistent. The larger the $k$, the more inconsistent the separation. We choose $k$ to vary in $\{1, 5, 10, 100\}$. 
	Figure \ref{fig:kappa} draws $K(\bx)$ when $k=5$ and Figure \ref{fig:data} illustrates the data distribution with 2000 samples when $k=100$.

	On the synthetic data, we train two classifiers: the regular DNN classifier ($\tilde{f}_n$) and the localized DNN classifier ($\hat{f}_n$). To be more specific, we choose $\tilde{f}_n$ to be a 3-layer ReLU network with width 250 and $\hat{f}_n$ to be the composition of $M=5$ local ReLU classifiers, each with depth 3 and width 100. The total number of weights for $\tilde{f}_n$ is slightly larger than that of $\hat{f}_n$. 
	As a surrogate to the hard-to-optimize 0-1 loss, we choose cross-entropy as the training loss, whose effectiveness in recovering the 0-1 loss minimizer has been investigated \citep{kohler2020statistical}. 
	For each setup, we run 10 replications and monitor the test accuracy on unseen test data. The results are presented in Figure \ref{fig:acc}, where we can see that when $k=1$, both regular and localized classifiers have comparable performances, but as $k$ gets larger, the gap widens and the localized classifier significantly outperforms. This is consistent with our theoretical findings.
	The superiority in convergence rate is better reflected in Figure \ref{fig:rate}, where we plot the log-log curve between the empirical excess risk and sample size when $k=5$. The slope for localized classifiers is significantly steeper, indicating faster convergence rate. 
	More experiment details can be found in Section \ref{sec:exp}.

		\begin{figure}
		\centering
		\hspace{-1cm}
		\subfigure{
			\label{fig:acc}
			\includegraphics[width=0.5\textwidth]{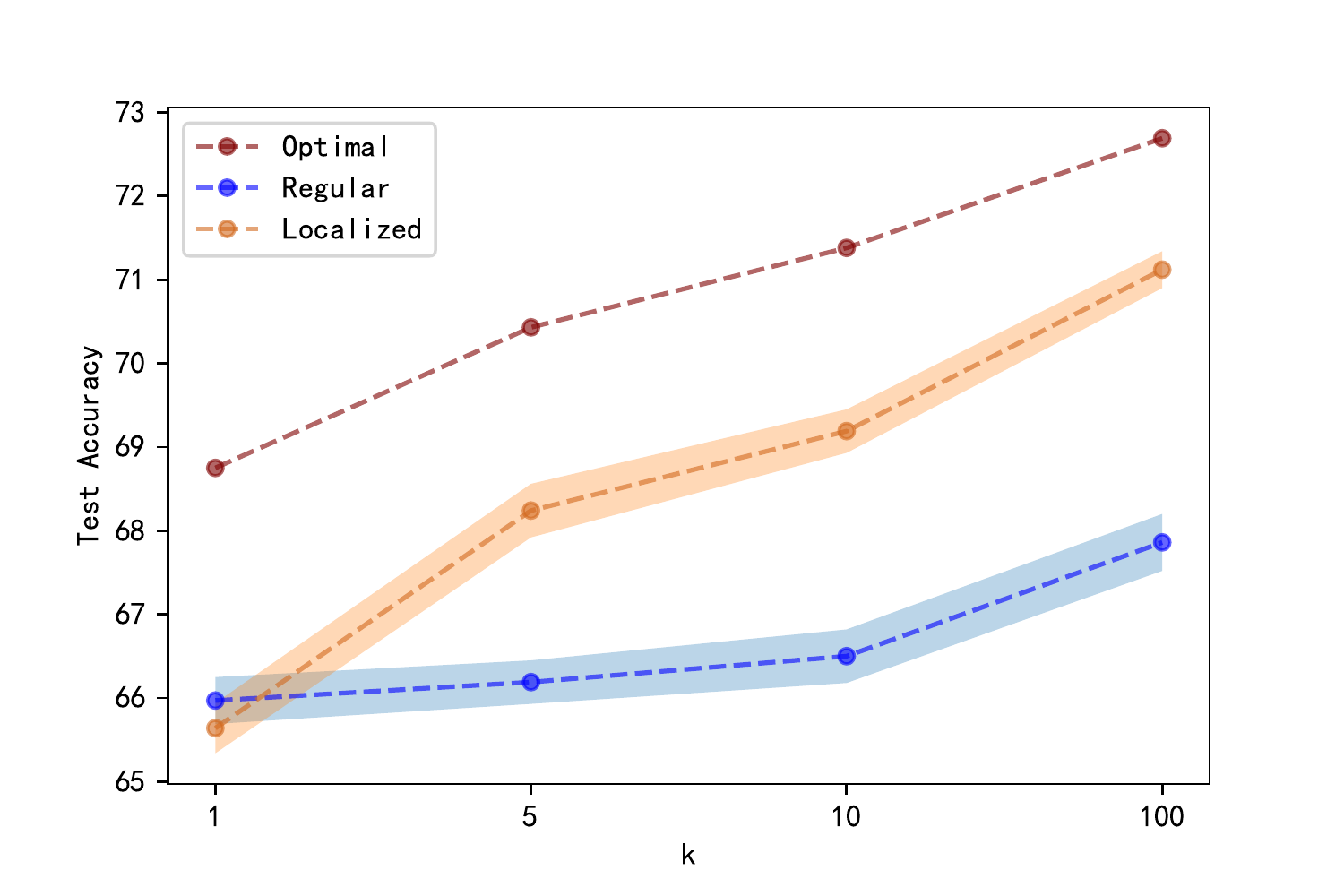}
		}\hspace{-0cm}
		\subfigure{
			\label{fig:rate}
			\includegraphics[width=0.5\textwidth]{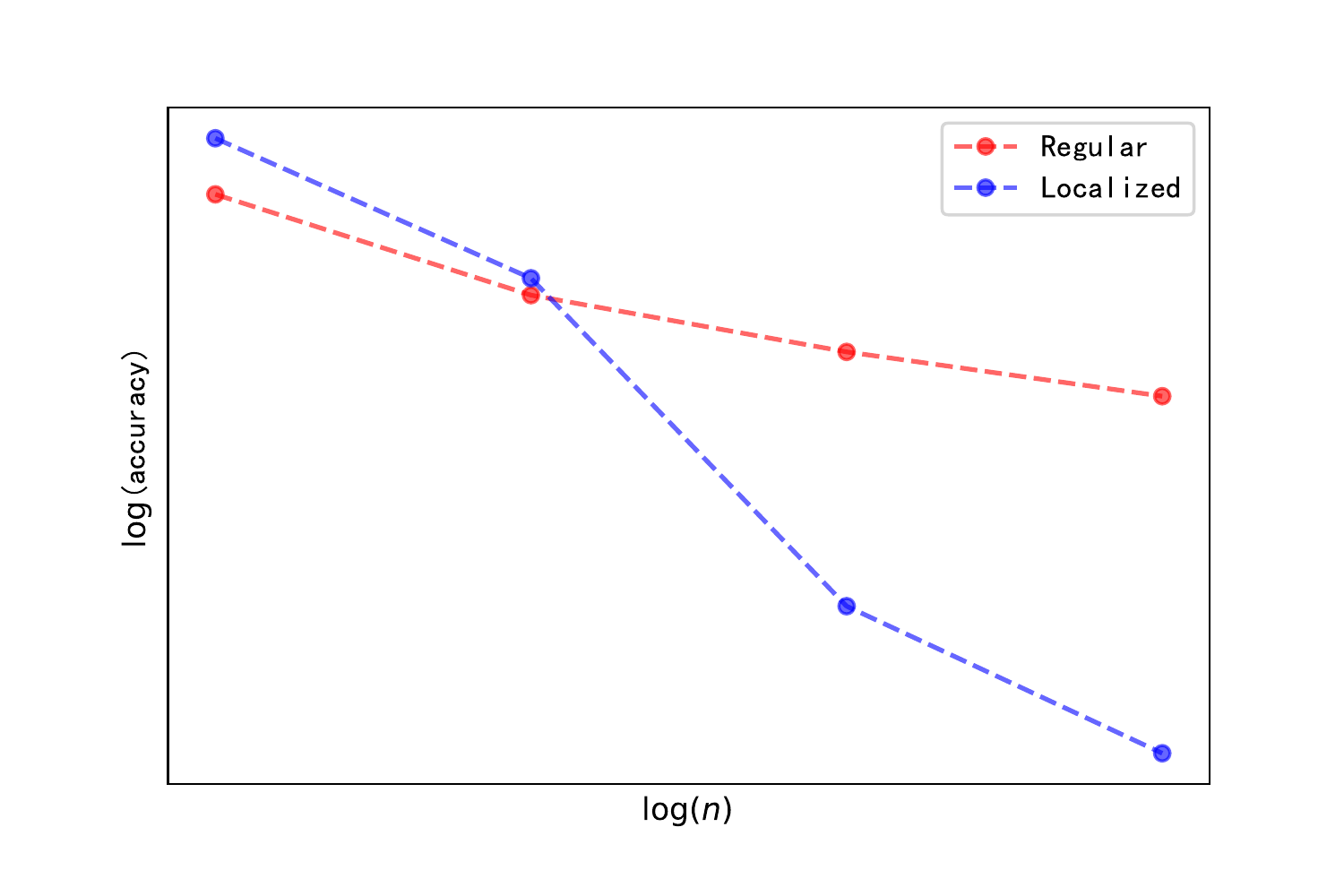}}
		\hspace{-1cm}
		\caption{ 
			(a) The test accuracy curves for different classifiers vs. $k=1, 5, 10, 100$. 
			As $k$ gets larger, the separation between classes are stronger but the consistency along the boundary is weaker. 
			The bands represent 2 times the standard deviation from 10 replications. The optimal classifier refers to the Bayes classifier using $\eta(\bx)$. 
			(b) The log-log curve of the empirical excess risk vs. sample size when $k=10$. The deeper the slope, the faster the convergence rate.}
	\end{figure}

	\section{Discussion}
	\label{sec:discussion}
	This work aims to close the gap between DNN's empirical success and its rate suboptimality in recovering smooth decision boundaries with Tsybakov's noise condition (N).
	Through a finer-grained analysis of the convergence rate, we uncover a potential source of the suboptimality, i.e., inconsistency of the boundary separation. 
	Correspondingly, we propose a new DNN classifier using a divide-and-conquer technique and a novel separation condition that is a localized version of the classical (N), under which optimal convergence rates are established for DNN classifiers with proper architectures. The established statistical optimality can adapt to low-dimensional structures underlying the data and in turn, may explain the absence of the ``curse of dimensionality" in practice.

	However, there are many limitations of this work. 
	First, the nonparametric framework analyzes properties of the empirical risk minimizer, e.g., convergence rate, sample complexity, etc., but doesn't provide guidance to practical optimization and how to find them. 
	Second, the proposed localized separation condition, like Tsybakov's noise condition, is of theoretical interest and cannot be verified for real data such as natural images. Similarly, this work only considers the 0-1 loss, which is also out of theoretical interest rather than practical relevance since it's the most natural and fundamental classification loss. 
	For future work, it is intriguing to study how DNN classifiers perform under popular surrogate losses such as hinge loss or cross-entropy. 
	Our numerical experiments conducted using cross-entropy indicate similar results may also hold for empirical surrogate loss minimizers. 
	Our theory can be further strengthened if we could relax the structural requirement for the DNN family to be more general, e.g, fully-connected, or extend the analysis to popular networks such as convolutional neural networks \citep{krizhevsky2012imagenet}, residual neural networks \citep{he2016deep}, and Transformers \citep{vaswani2017attention}.

	We believe that the proposed localized separation condition and the corresponding localized analysis are general and potentially transferable to other classification methods, as long as they can take advantage of the divide-and-conquer strategy. 
	We hope more investigations can be inspired along this line.

	\section{Technical details}\label{sec:proof}
	
	This paper studies classification under the smooth boundary fragment setting \eqref{eqn:boundary_fragment} with Tsybakov's noise condition \citep{mammen1999smooth}.
	First, we provide proof for Lemma \ref{lemma:m1}, which characterizes the relationship between the classical (N) and the proposed intermediate (N$^+$). 
	
	\noindent\textit{Proof of Lemma \ref{lemma:m1}}
	
	\begin{proof}
		Recall the definition of $K(\bx)$ and condition (M1) that there exists $\epsilon_0>0$ and $0<C_{\epsilon_0}<\infty$ such that for all $\bx\in \partial G^*$ and any $0<t<\epsilon_0$, 
		\[
		\frac{1}{C_{\epsilon_0}} \le \frac{m_{\bx}(t)}{|t|^{1/K(\bx)}} \le C_{\epsilon_0}.
		\]
		The key idea is to treat $\bx_{-d}$ and $x_d$ separately. For any $0<t<{\epsilon_0^{1/\kappa^+}}/{C_{\epsilon_0}}$, we can write
		\begin{align*}
			\QQ(\bx\in[0,1]^d: |p(\bx)-q(\bx)|\le t) &= \int_{[0,1]^{d-1}}\int_0^1\II\{|p(\bx)-q(\bx)|\le t\}dx_d d\bx_{-d} \\
			&\le
			\int_{[0,1]^{d-1}}\int_{-1}^1\II\{m_{\bx_{-d}}(u)\le t\}du d\bx_{-d}\\
			&\le
			2\int_{[0,1]^{d-1}}\int_{0}^1\II\{\frac{u^{1/K(\bx)}}{C_{\epsilon_0}} \le t\}du d\bx_{-d}\\
			&\le
			2\int_{[0,1]^{d-1}}\int_{0}^1\II\{{u^{1/\kappa^-}} \le tC_{\epsilon_0}\}du d\bx_{-d}\\
			&=
			2\int_{[0,1]^{d-1}}{(C_{\epsilon_0}t)^{\kappa^-}} d\bx_{-d}\\
			&= 2{C_{\epsilon_0}t}^{\kappa^-}.
		\end{align*}
		The first equality is by definition. The second inequality is due to condition (M1), which indicates that locally, $m_{\bx}\ge {\epsilon_0^{1/K(x)}}/{C_{\epsilon_0}}\ge {\epsilon_0^{1/\kappa^+}}/{C_{\epsilon_0}}.$
		Now consider (N$^+$) and a set $G\subset[0,1]^d$ containing part of the decision boundary.
		Notice that only lower bound is concerned so if suffices to prove the inequality for a subset $\subset G$. Without loss of generality, let the subset be $[0, \epsilon_G]^d$.
		The analysis is similar to that in $[0,1]$ and 
		for $t$ small enough we have
		\begin{align*}
			\QQ(\bx\in G: |p(\bx)-q(\bx)|\le t) &= 
			\int_{[0,\epsilon_G]^{d-1}}\int_{-\epsilon_G}^{\epsilon_G}\II\{m_{\bx_{-d}}(u)\le t\}du d\bx_{-d}\\
			&\ge
			\int_{[0,\epsilon_G]^{d-1}}\int_{-{\epsilon_G}}^{\epsilon_G}\II\{C_{\epsilon_0}{u^{1/K(\bx)}} \le t\}du d\bx_{-d}\\
			&\ge
			\int_{[0,\epsilon_G]^{d-1}}\int_{-{\epsilon_G}}^{\epsilon_G}\II\{{u^{1/\kappa^+}} \le t/C_{\epsilon_0}\}du d\bx_{-d}\\
			&\ge C_G t^{\kappa^+},
		\end{align*}
		where $C_G$ is some constant depending on the set $G$ and $C_\epsilon$. 
	\end{proof}
	
	Next, we prove some useful lemmas utilizing the proposed localized separation condition (M1,M2).  
	Let $G_f :=\{\bx\in\cX: f(\bx_d)-x_d\ge 0\}$. Then
	we have the following lemma characterizing the relationship between $d_\triangle$ and $d_{p,q}$.
	
	\begin{lemma}
		\label{lemma:d_ineq}
		Under assumption (M1), further assume on some $D\subset\cX$, $0<\kappa^-\le K(\bx)$ for all $\bx\in D$. For any set $G=G_{f} \subset D$ 
		satisfying $\|f-f^*\|_\infty\le \epsilon_0$, the following inequality holds
		\[
		d_\triangle(G,G^*)^{\frac{\kappa^-+1}{\kappa^-}}\lesssim d_{p,q}(G,G^*).
		\]
	\end{lemma}
	\begin{proof}
		Let $\delta(\bx_{-d}):=|f(\bx_{-d})-f^*(\bx_{-d})|\le \epsilon_0$.
		Consider $G\triangle G^*$ in dimension $x_d$ and $\bx_{-d}$ separately and write $G\triangle G^*=\rbr{(G\triangle G^*)_{-d},(G\triangle G^*)_d}$. Then
		\begin{align*}
			d_\triangle(G,G^*) &= \int_{(G\triangle G^*)_{-d}}\int_{(G\triangle G^*)_d}dx_dd\bx_{-d}\\
			&=\int_{(G\triangle G^*)_{-d}}\delta(\bx_{-d})d\bx_{-d}
		\end{align*}
		Applying assumption (M1) and Jensen's inequality yields
		\begin{align*}
			d_{p,q}(G,G^*) &= \int_{G\triangle G^*} |p(\bx)-q(\bx)|d\bx\\
			&= 
			\int_{(G\triangle G^*)_{-d}}\int_{0}^{\delta({\bx_{-d}})} m_{\bx}(t)dt d\bx_{-d}\\
			&\ge 
			\int_{(G\triangle G^*)_{-d}}\int_{0}^{\delta({\bx_{-d}})} \frac{1}{C_{\epsilon_0}} t^{1/\kappa^-}dt d\bx_{-d}\\
			& \ge 
			\frac{1}{C_{\epsilon_0}(1+1/\kappa^-)}\int_{(G\triangle G^*)_{-d}}{\delta({\bx_{-d}})^{\frac{\kappa^- +1}{\kappa^-}}}d\bx_{-d}\\
			& \ge
			\frac{1}{C_{\epsilon_0}(1+1/\kappa^-)} d_{\triangle}(G,G^*)^{\frac{\kappa^- +1}{\kappa^-}}
		\end{align*}
	\end{proof}
	
	\noindent\textit{Convergence rate proof road map:}
	As stated before, the excess risk can be decomposed into the stochastic (empirical) error and the approximation error. 
	As DNN classifiers get larger, the variance becomes larger while the bias gets smaller. 
	An optimal trade-off can be achieved by carefully choosing the size of the DNN function space, or equally, the approximation error.  
	The better the characterization of the two types of error, the tighter the convergence upper bound. 
	In the following, we investigate these two aspects separately.

	\subsection{Approximation error}
	DNNs are universal approximators \citep{cybenko1989approximations}. 
	Approximating the smooth function $f^*\in\cH(d, \beta)$ in (\ref{eqn:boundary_fragment}) using a DNN family $\tilde{\cF}_n$, has been well-established in literature. Note that here we use $\tilde{\cF}$ to distinguish from the overall DNN classifier family $\cF$. 
	Among others \citep{yarotsky2017error, kohler2019rate}, 
	\cite{lu2020deep} show that if $\tilde{\cF}_n$ is large enough with depth $\tilde{L}=O(L\log(L))$ and width $\tilde{N}=O(N\log(N))$, the $L_\infty$-approximation error is in the order of $O(N^{-2\beta/d}L^{-2\beta/d})$; 
	\cite{schmidt2020nonparametric} proves that for any $\epsilon>0$, there exists a neural network $\tilde{f}_n\in\tilde{\cF}_n$ with $\tilde{L}_n = O(\log n)$ layers and
	$\tilde{S}_n = O(\epsilon^{-d/\beta}\log n)$ non-zero weights such that
	$\|\tilde{f}_n - f^*\|_\infty \le \epsilon$. The detailed theorems are listed below and they cover all the cases considered in this paper.
	It's worth emphasizing that approximation of smooth functions is not the focus of this work. We aim to better utilize existing approximation results to sharpen the bias bound of the excess risk.
	\begin{lemma}[Theorem 1.1 in \citep{lu2020deep}]
		\label{lemma:approximation}
		Given a smooth function $f_0\in  \cH(d, \beta)$ with $\beta\in \NN^+$, for any $N,L\in \NN^+$, there exists a function $f$ implemented by a ReLU DNN with width $ C_1(N+2)\log_2(8N)$ and depth $C_2(L+2)\log_2(4L)+2d$ such that 
		\begin{equation*}
			\|f-f_0\|_{L^\infty([0,1]^d)}\le C_3 \|f\|_{C^s([0, 1]^d)} N^{-2s/d}L^{-2s/d},
		\end{equation*}
		where $C_1=17s^{d+1}3^dd$, $C_2=18s^2$, and $C_3=85(s+1)^d8^s$.
	\end{lemma}
	\begin{lemma}[Approximation Part of Theorem 1 in \citet{schmidt2020nonparametric}]
		\label{smooth}
		Consider the $d$-variate nonparametric regression model for composite regression function $f_0$ in the class $\cH(l, \bd, \bt, \bbeta, R).$ There exists $\tilde{f}_n$ in the network class $\cF_n^{\dnn}(L_n, N_n, S_n, B_n, F_n)$ with $L_n\lesssim \log_2 n$, 
		$B_n=1$,
		$F_n \geq \max(R,1)$, 
		$$N_n\lesssim \max_{i=0,\cdots,q}n^{\frac{t_i}{2\bar{\beta}_i+t_i}}, \quad S_n \lesssim \max_{i=0,\cdots,q}n^{\frac{t_i}{2\bar{\beta}_i+t_i}} \log n.$$    
		such that 
		\begin{align*}
			\|\hat{f}_n-f_0\|_\infty\lesssim n^{-{\beta^{*}}/({2\beta^{*}+d^*})}.
		\end{align*}
	\end{lemma}
	Notice that the approximation error bounds in the above lemmas are almost rate-optimal. Detailed descriptions and proofs can be found in the referenced papers. 
	Other neural network approximation results can also be potentially used in our localized analysis as long as the approximation rate is optimal up to a logarithmic term. 
	
	Two types of smooth functions are considered, one is general H\"older smooth function and the other is compositional smooth function. 
	Since the former is a special case of the latter, most of the proofs are done on the more general compositional smooth setting (\ref{comp}).

	\subsection{Stochastic error}
	The techniques for controlling the stochastic error are from literature on \textit{empirical process} where covering/bracketing number/entropy are widely used to upper bound the empirical error. More details can be found in \cite{sara2000, kosorok2007}.
	
	However, there are difficulties when it comes to DNNs. A subtle difference in analyzing neural networks lies in their complexity measurement. To be more specific, for some real-value function space $\cF$, let the covering entropy be $H(\delta,\cF, L_\infty)$. The typical entropy bound is of the form $A\delta^{-\rho}$ where $A,\rho$ are some constants, e.g., for $d$-variant $\beta$-smooth functions, $\rho = d/\beta$. 
	In comparison, that for DNNs is of the form $A\log(1/\delta)$ where $A=O(N^2L^2)$, which diverges with network size. 
	\begin{lemma}[Lemma 5 of \cite{schmidt2020nonparametric}; Lemma 3 of \cite{suzuki2018adaptivity}]
		\label{lemma:entropy}
		For any $\delta>0$, a ReLU network family $\cF$ satisfies
		\begin{align*}
			\log \cN(\delta, \cF(L, N, S, B), \|\cdot\|_{\infty}) 
			\le 2L(S+1)\log(\delta^{-1}(L+1)(N+1)(B\vee 1)).
		\end{align*}
	\end{lemma}
	
	The difference makes existing results not directly applicable and non-trivial modifications have to be made. Below, we prove some empirical error bounds specifically for DNNs. 
	Following the notations from \cite{sara2000} and \cite{tsybakov2004optimal},
	let 
	$$v_n(h)=\sqrt{n}\int h(\bx) d(P_n-P),$$
	where $P$ denotes the data distribution, i.e. $\bx\sim P$ and $P_n$ denotes the empirical distribution of $\bx_1,\cdots,\bx_n$. 
	\begin{lemma}
		[Theorem 5.11 in \cite{sara2000}]
		\label{lemma:511}
		For some function space $\cH$ with $\sup_{h\in\cH}\|h(\bx)\|_\infty\le K$ and $\sup_{h\in\cH}\|h(\bx)\|_{L_2(P)}\le R$.
		If $a>0$ satisfies: 
		(1) $a\le C_1\sqrt{n}R^2/K$; 
		(2) $a\le 8\sqrt{n}R$;
		\[
		(3)\quad a\ge C_0\left(\int_{a/64\sqrt{n}}^R H_B^{1/2}(u,\cF,L_2(P))du\vee R\right);
		\]
		and (4) $C_0^2\ge C^2(C_1+1)$. Then 
		\[
		\PP\left(\sup_{h\in\cH}\abr{\sqrt{n}\int hd(P_n-P)}\ge a\right)\le C\exp\rbr{-\frac{a^2}{C^2(C_1+1)R^2}},
		\]
		where $P_n$ is the empirical counterpart of $P$.
	\end{lemma}
	
	The next lemma investigates the modulus of continuity of the empirical process. It's similar to Lemma 5.13 in \cite{sara2000} but with a key difference in the entropy assumption (\ref{eqn:hb}), where the entropy bound contains $n$. 
	\begin{lemma}
		\label{lemma:geer}
		For a probability measure $P$, let $\cH_n$ be a class of uniformly bounded (by 1) functions $h$ in $L_2(P)$ depending on $n$. Suppose that the $\delta$-entropy with bracketing satisfies for all $0<\delta<1$ small enough,
		the inequality
		\begin{align}
			\label{eqn:hb}
			H_B(\delta,\cH_n, L_2(P))\le A_n\log (1/\delta),
		\end{align}
		where $0<A_n=o(n)$.
		Let $h_{0n}$ be a fixed element in $\cH_n$.
		Let $\cH_n(\delta) = \{h_n\in\cH_n: \|h_n-h_{0n}\|_{L_2(P)}\le \delta\}$.
		Then there exist constants $D_1>0, D_2>0$ such that for a sequence of i.i.d. random variables $\bx_1,\cdots,\bx_n$ with probability distribution $P$, it holds that for all $T$ large enough,
		\begin{align*}
			&\PP\left(\sup_{h_n\in\cH_n(\sqrt{{A_n}/{n}})}\abr{\int (h_n-h_{0n})d(P_n-P)}\ge T{\frac{A_n}{{n}}}\right)\\
			&\le C\exp\left(-\frac{TA_n}{8C^2}\right)
		\end{align*}
		and for $n$ large enough,
		\begin{align*}
			& \PP\left(\sup_{\substack{h_n\in\cH_n;\\ \|h_n-h_{0n}\|> \sqrt{{A_n}/{n}}}}\frac{|v_n(h_n)-v_n(h_{0n})|}{A_n^{1/2}\|h_n-h_{n0}\|}>D_1 x\right)\\
			&\le D_2 e^{-A_n x}
		\end{align*}
		for all $x\ge 1$. 
	\end{lemma}
	\begin{proof}
		The main tool for the proof is Lemma \ref{lemma:511}. Replace $\cH$ with $\cH_n(\delta)$ in Lemma \ref{lemma:511} and take $K=4$, $R = \sqrt{2}\delta$ and 
		$a = \frac{1}{2}C_1A_n^{1/2}\delta$, with $C_1 = 2\sqrt{2}C_0$. 
		Then (1) is satisfied if 
		\begin{equation}
			\label{eqn:an}
			\delta\ge \sqrt{\frac{A_n}{n}},
		\end{equation} 
		under which, (2) and (3) is trivially satisfied when $n$ is large enough. 
		Choosing $C_0$ sufficiently large will ensure (4). Thus, for all $\delta$ satisfying (\ref{eqn:an}), we have
		\begin{align*}
			&\PP\left(\sup_{h_n\in\cH_n(\delta)}\abr{\sqrt{n}\int (h_n-h_{0n})d(P_n-P)}\ge \frac{C_1}{2}{A_n}^{1/2}\delta\right)\\
			&\le C\exp\left(-\frac{C_1A_n}{16C^2}\right)
		\end{align*}
		Let $B = \min\{b>1: 2^{-b}\le \sqrt{{A_n}/{n}}\}$ and apply the peeling device. Then,
		\begin{align*}
			&\PP\rbr{\sup_{\substack{h_n\in\cH_n;\\ \|h_n-h_{n0}\|> \sqrt{{A_n}/{n}}}}\frac{\abr{\sqrt{n}\int (h_n-h_{n0})d(P_n-P)}}{A_n^{1/2}\|h_n-h_{n0}\|}\ge \frac{C_1}{2}}\\
			&\le\sum_{b=0}^B\PP\rbr{\sup_{h_n\in\cH_n(2^{-b})}{\abr{\sqrt{n}\int (h_n-h_{n0})d(P_n-P)}}\ge \frac{C_1}{2}{A_n}^{1/2}(2^{-b}) }\\
			&\le\sum_{b=0}^B C\exp\left(-\frac{C_1A_n }{16C^2}\right)\le 2C(\log n)\exp\rbr{-\frac{C_1A_n}{16C^2}},
		\end{align*}
		if $C_1A_n$ is sufficiently large.
	\end{proof}
	
	Now that we have introduced necessary tools to analyze both the approximation error and the stochastic error, we continue to prove the main theorems in this work.

	\subsection{Localized convergence analysis}
	\label{sec:rate_local}
	Let $\rho=\frac{d-1}{\beta}$ in the regular smooth case and $\rho^*=\frac{d^*}{\beta^{*}}$ for the compositional smooth case, which denote the complexity measurements for the target function space.  
	
	\noindent\textit{Proof of Theorem \ref{thm:local_estimator}}
	
	\begin{proof}
		For ease of notation, we will write $G_f$ and its defining function $f$ interchangeably. 
		For any $\epsilon>0$, by construction, we can find $\tilde{f}_n\in\tilde{\cF}_n$ such that 
		$\|\tilde{f}_n - f^*\|_\infty \le \epsilon$. Within $D_{\bj_{-d}}$, the 0-1 loss can be bounded as  
		\begin{align*}
			d_{\bj_{-d}}(\tilde{f}_{n},f^*)&=\int_{D_{\bj_{-d}}:G_{\tilde{f}_{n,{\bj_{-d}}}}\triangle G_{f^*}}|p(\bx)-q(\bx)|d\bx\\
			&\le \int_{D_{\bj_{-d}}}\int_0^{\epsilon} m_{\bx}(t)dt d\bx_{-d}\\
			&\le C_{\epsilon_0}\int_{D_{\bj_{-d}}}\int_0^{\epsilon} t^{1/K(\bx)}dt d\bx_{-d}\\
			& \le {\frac{C_{\epsilon_0}}{M^{d-1}(1+1/\kappa^+)}\epsilon^{\frac{\kappa^+ +1}{\kappa^+}}}
		\end{align*}
		Since $\hat{f}_{n, {\bj_{-d}}}$ is the empirical risk minimizer within $\tilde{\cF}_n$, we have $R_{n,{\bj_{-d}}}(\hat{f}_{n, {\bj_{-d}}})\le R_{n,{\bj_{-d}}}({\tilde{f}_{n}})$.
		Therefore,
		\begin{align*}
			d_{{\bj_{-d}}}(\hat{f}_{n, {\bj_{-d}}},f^*) 
			\le& 
			d_{{\bj_{-d}}}(\tilde{f}_{n},f^*) + [R_{n,{\bj_{-d}}}(\tilde{f}_{n})-R_{n,{\bj_{-d}}}(f^*)-d_{\bj_{-d}}(\tilde{f}_{n}, f^*)]\\
			& +
			[R_{n,{\bj_{-d}}}(f^*)-R_{n,{\bj_{-d}}}(\hat{f}_{n,{\bj_{-d}}})+d_{\bj_{-d}}(\hat{f}_{n, {\bj_{-d}}}, f^*)]\\
			:\le &
			\frac{C_{\epsilon_0}}{M^{d-1}(1+1/\kappa^+)}\epsilon^{\frac{\kappa^+ +1}{\kappa^+}} +
			I(\tilde{f}_{n}, f^*) + I(\hat{f}_{n, {\bj_{-d}}}, f^*).
		\end{align*}
		For $I(\tilde{f}_{n}, f^*)$, by Lemma \ref{lemma:geer}, we have
		\begin{align*}
			I(\tilde{f}_{n}, f^*)
			\le&  
			\sup_{\substack{f\in \tilde{\cF}_n: \|f-f^*\|_1 \\
					\le \sqrt{{A_n}/{n}}}}\abr{R_{n,{\bj_{-d}}}(f)-R_{n,{\bj_{-d}}}(f^*)-d_{\bj_{-d}}(f,f^*)}+\\ 
			& 
			\sqrt{\frac{A_n d_\triangle(\tilde{f}_{n}, f^*)}{{n}}}\sup_{\substack{f\in \tilde{\cF}_n: \|f-f^*\|_1 \\> \sqrt{{A_n}/{n}}}}\frac{\sqrt{n}\abr{R_{n,{\bj_{-d}}}(f)-R_{n,{\bj_{-d}}}(f^*)-d_{{\bj_{-d}}}(f,f^*)}}{\sqrt{A_n d_\triangle({f}, f^*)}}\\
			&= O_{\PP}\rbr{\frac{A_n}{{n}}} + \sqrt{\frac{A_n d_\triangle(\tilde{f}_{n}, f^*)}{{n}}} \ O_{\PP}(1),
		\end{align*}
		where $A_n$ is from the assumption (\ref{eqn:hb}).
		Similarly for $I(\hat{f}_{n, {\bj_{-d}}}, f^*)$, we have
		\begin{align*}
			I(\hat{f}_{n, {\bj_{-d}}}, f^*)
			&= O_{\PP}\rbr{\frac{A_n}{{n}}} + \sqrt{\frac{A_n d_\triangle(\hat{f}_{n,{\bj_{-d}}}, f^*)}{{n}}} \ O_{\PP}(1).
		\end{align*}
		By construction, $d_\triangle(\tilde{f}_{n}, f^*)\le \epsilon$. Hence
		\begin{align*}
			d_{{\bj_{-d}}}(\hat{f}_{n, {\bj_{-d}}},f^*) 
			\le &
			\frac{C_{\epsilon_0}}{M^{d-1}(1+1/\kappa^+)}\epsilon^{\frac{\kappa^+ +1}{\kappa^+}} +
			O_{\PP}\rbr{\frac{A_n}{{n}}} + \\
			& \sqrt{\frac{A_n \rbr{d_\triangle(\hat{f}_{n,{\bj_{-d}}}, f^*)+\epsilon}}{{n}}} \ O_{\PP}(1).
		\end{align*}
		The last term dominates the second term if $\epsilon\gg n^{-1/(1+\rho)}$. 
		Let $\epsilon\gtrsim n^{-\frac{1}{1+\rho+2/\kappa^+}}$. Then the first term $\epsilon^{\frac{\kappa^++1}{\kappa^+}}$ dominates the $\sqrt{A_n\epsilon/n}$ term in the last term.  
		Omitting the approximation error, i.e. $$\epsilon^{\frac{\kappa^+ +1}{\kappa^+}}\lesssim\sqrt{\frac{A_n}{{n}}}d^{1/2}_\triangle(\hat{f}_{n,{\bj_{-d}}}, f^*),$$ by Lemma \ref{lemma:d_ineq} we have
		\begin{align*}
			d_{{\bj_{-d}}}(\hat{f}_{n,{\bj_{-d}}}, f^*) &\le \sqrt{\frac{A_n}{{n}}}d^{1/2}_\triangle(\hat{f}_{n,{\bj_{-d}}}, f^*) \ O_{\PP}(1)\\
			&\le \sqrt{\frac{A_n}{{n}}} d_{{\bj_{-d}}}(\hat{f}_{n,{\bj_{-d}}}, f^*)^{\frac{\kappa^-}{2(\kappa^-+1)}}\ O_{\PP}(1),
		\end{align*}
		which simplifies to
		\begin{align*}
			d_{{\bj_{-d}}}(\hat{f}_{n,{\bj_{-d}}}, f^*) = O_{\PP}\rbr{{\frac{{A_n}}{n}}}^{\frac{\kappa^-+1}{\kappa^-+2}}.
		\end{align*}
		From Lemma \ref{lemma:entropy}, we know $A_n=O(\tilde{N}_n^2\tilde{L}_n^2)$. Combining with Lemma \ref{lemma:approximation} and Lemma \ref{smooth}, we know that
		$A_n = \tilde{O}(\epsilon^{-\rho})$ in both cases. 
		Balancing the approximation error and the empirical error, choose 
		\begin{equation}
			\label{eqn:epsilon_tilde}
			\epsilon \asymp  {n^{-\frac{\kappa^+(\kappa^-+1)}{(\kappa^-+2)(\kappa^+ +1)+\rho\kappa^+(\kappa^-+1)}}}.
		\end{equation}
		In this case, Lemma \ref{lemma:approximation} requires $\tilde{N}_n\tilde{L}_n\asymp \epsilon^{1/2}\log^2(n)$ and $A_n\asymp \epsilon^{-\rho}\log^4(n)$. Hence
		\begin{align*}
			\EE \ d_{\bj_{-d}}(\hat{f}_{n,{\bj_{-d}}},f^*) = {O}\rbr{{\frac{1}{n}}}^{\frac{\kappa^-+1}{\kappa^-+2+ \kappa^+\rho\rbr{\frac{\kappa^-+1}{\kappa^+ +1}}}}\cdot \log^4(n).
		\end{align*}
	\end{proof}
	
	\begin{remark}[Size of $\tilde{\cF}_n$]
		$\tilde{\cF}_n$ should have the capacity to approximate $f^*$ well such that the $L_\infty$-error within the local region can be bounded by $\epsilon$ as chosen in (\ref{eqn:epsilon_tilde}). DNN approximating smooth functions have been well-studied and the DNN structure is pretty flexible. For example, following
		\cite{schmidt2020nonparametric}, the DNN family should be large enough such that $\tilde{L}_n \asymp \log(n)$ and $\tilde{N}_n\asymp \epsilon^{-\rho/2}$. \cite{lu2020deep} provide a more general approximation result and the DNN family only have to satisfy $\tilde{N}_n\tilde{L}_n\asymp \epsilon^{-\rho/2}\log^2 n$. 
	\end{remark}

	\noindent\textit{Proof of Lemma \ref{lemma:thm1}}~
	\begin{proof}
		As illustrated in the proof of Lemma \ref{lemma:m1}, condition (N) and (N$^+$) essentially imply that $K(\bx)=\kappa$ along the entire decision boundary, which indicates $\kappa^-=\kappa^+$. Following the same proof technique of Theorem \ref{thm:local_estimator}, we can see that the optimal rate convergence rate can be achieved. 
	\end{proof}

	\begin{figure}[t]
		\centering
		\includegraphics[width=0.5\textwidth]{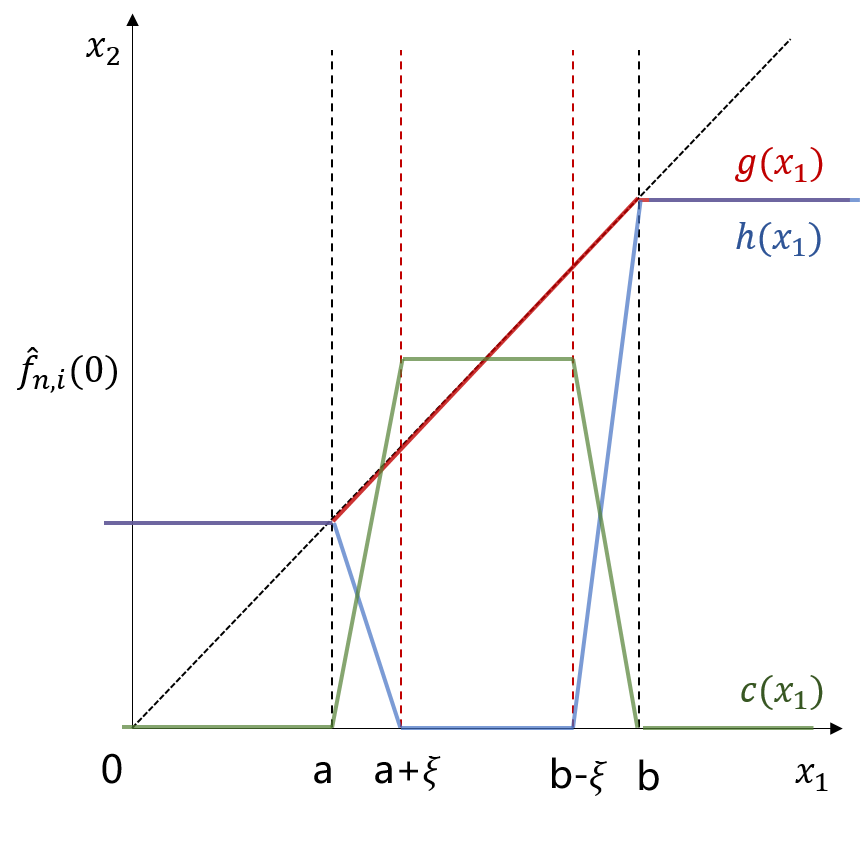}
		\caption{Illustration of the constructed functions $g,h,c$ in $d=2$ case.}
		\label{fig:ab}
	\end{figure}

	\subsection{Verification of (P1,P2,P3)}
	\label{sec:p13}
	Let's first consider the $d=2$ case and focus on some region $D_i=\{(x_1,x_2)\in [0,1]^2: x_1\in (a,b)\}$ with $b-a>2\xi$. Let $f_n\in\tilde{\cF}$ be any DNN. 
	Define three continuous piecewise linear functions 
	\[
	g(x_1) =
	\begin{cases}
		x_1 & \text{if $a\le x_1\le b$} \\
		a & \text{if $x_1<a$} \\
		b & \text{if $x_1>b$}
	\end{cases}, \quad
	c(x_1) =
	\begin{cases}
		{f}_{n}(0) & \text{if $a+\xi\le x_1\le b-\xi$} \\
		0 & \text{if $x_1<a$ or $x_1>b$} \\
		\text{linear transition} & \text{else}
	\end{cases}
	\]
	and 
	\[
	h(x_1) =
	\begin{cases}
		0 & \text{if $a+\xi\le x_1\le b-\xi$} \\
		a & \text{if $x_1<a$} \\
		b & \text{if $x_1>b$}\\
		\text{linear transition} & \text{else}
	\end{cases}
	\]
	Linear transition means linking the end points with a line segment.
	The constructed piecewise linear functions are illustrated in Figure \ref{fig:ab}. 
	Let ${f}^+_{n,i}(x_1) := {f}_{n}(g(x_1)) -{f}_{n}(h(x_1)) + c(x_1)$. Then, it's easy to verify that 
	\[
	{f}^+_{n,i}(x_1) =
	\begin{cases}
		{f}_{n}(x_1)& \text{if $a+\xi\le x_1\le b-\xi$} \\
		0 & \text{if $x_1<a$ or $x_1>b$} \\
		\text{piecewise linear} & \text{else}
	\end{cases}
	\]
	Therefore, (P1) and (P2) hold and now we evaluate (P3). 
	The constructed $g, h, c$ are all piecewise linear functions with at most 5 pieces. By Theorem 2.2 in \cite{arora2016understanding}, they can all be represented by two-layer ReLU neural networks with width at most 5. 
	${f}^+_{n,i}(x_1)$ is constructed by composition and addition of ReLU networks, which correspond to stacking more layers and expanding the width respectively. 
	Easy to see that ${f}^+_{n,i}(x_1)$ satisfies (P3).
	
	In the $d>2$ case, we can make similar constructions. Consider some region $D_{\bj_{-d}}$ and denote $D^\circ_{\bj_{-d}}:=D_{\bj_{-d}}\backslash D_\xi$. For each of the dimensions $x_1,\ldots, x_{d-1}$, we can define $g_i(x_i), h_i(x_i), c_i(x_i)$ separately as in the $d=2$ case.
	
	Let $g(\bx_{-d}) = (g_1(x_1),\ldots, g_{d-1}(x_{d-1}))$, $h(\bx_{-d}) = (h_1(x_1),\ldots, h_{d-1}(x_{d-1}))$, $c(\bx_{-d}) = (c_1(x_1),\ldots, c_{d-1}(x_{d-1}))$ and  
	${f}^+_{n,\bj_{-d}} = \rbr{{f}_{n}\circ g -{f}_{n}\circ h + c}$. Then, it's easy to verify that 
	\[
	{f}^+_{n,\bj_{-d}}(\bx_{-d}) =
	\begin{cases}
		{f}_{n}(\bx_{-d})& \text{if $\bx_{-d}\in D^\circ_{\bj_{-d}}$} \\
		0 & \text{if $\bx_{-d}\notin D_{\bj_{-d}}$} \\
		\text{piecewise linear} & \text{else}
	\end{cases}
	\]
	Thus, (P1) and (P2) hold. For (P3), notice that $g(\bx_{-d})$ can be viewed as a ReLU neural network with the same depth as $g_i(x_i)$ but $(d-1)$-times the width. 
	
	\subsection{Global convergence analysis}
	\label{sec:rate_global}
	\noindent\textit{Proof of Theorem \ref{thm:global_estimator_eff}}~
	\begin{proof}
		Choose $\xi=1/n^2$, $M=\log n$.
		Notice that $nM\xi(d-1) \to 0$, i.e., $\PP(E_\xi)\to 1$ as $n\to\infty$ for any $d \ge 2$. 
		Without loss of generality, assume $n=\Omega(\epsilon_0^{-(1+\rho)})$. Thus $\xi$ is smaller than $\epsilon_0$. 
		Let 
		\[\kappa^-_{\bj_{-d}}:=\min_{\bx\in D_{\bj_{-d}}} K(\bx)
		\quad\mbox{and}\quad \kappa^+_{\bj_{-d}}:=\max_{\bx\in D_{\bj_{-d}}} K(\bx).\] 
		Since $R_{n,\bj_{-d}}(\hat{f}_{n})=R_{n,\bj_{-d}}(\hat{f}_{n, \bj_{-d}})\le R_{n,\bj_{-d}}(\tilde{f}_{n})$ for any $\bj_{-d}\in J_M$ as in (\ref{eqn:local}),
		Theorem \ref{thm:local_estimator} yields that
		\begin{align*}
			\sup_{f^*\in \cF(d^*,\beta^{*}) }\EE(R_{\bj_{-d}}(\hat{f}_{n})-R_{\bj_{-d}}({f^*})) 
			\lesssim  \rbr{\frac{1}{n}}^{\frac{\kappa^-_{\bj_{-d}}+1 }{\kappa^-_{\bj_{-d}}+2 +\rbr{\frac{\kappa^-_{\bj_{-d}}+1}{\kappa^+_{\bj_{-d}} +1}}\rho\kappa^+_{\bj_{-d}}}}\cdot \log^4(n).
		\end{align*}
		Then, the overall 0-1 loss excess risk can be decomposed as
		\begin{align*}
			\sup_{f^*\in \cF(d^*,\beta^{*}) }\EE(R(\hat{f}_{n})-R({f^*})) &\le 
			\sum_{\bj_{-d}\in J_M} \sup_{f^*\in \cF(d^*,\beta^{*}) }\EE(R_{\bj_{-d}}(\hat{f}_n)-R_{\bj_{-d}}({f^*}))\\
			&\lesssim  \sum_{\bj_{-d}\in J_M} \rbr{\frac{1}{n}}^{\frac{\kappa^-_{\bj_{-d}}+1 }{\kappa^-_{\bj_{-d}}+2 +\rbr{\frac{\kappa^-_{\bj_{-d}}+1}{\kappa^+_{\bj_{-d}} +1}}\rho\kappa^+_{\bj_{-d}}}}\cdot \log^4(n).
		\end{align*}
		By assumption (M2), we can write for any $\bj_{-d} \in J_M$ that
		\begin{align*}
			\rbr{\frac{1}{n}}^{\frac{\kappa^-_{\bj_{-d}}+1}{\kappa^-_{\bj_{-d}}+2+\rbr{\frac{\kappa^-_{\bj_{-d}}+1}{\kappa^+_{\bj_{-d}} +1}}\rho\kappa^+_{\bj_{-d}}}} 
			&=  
			\rbr{\frac{1}{n}}^{\frac{\kappa^-_{\bj_{-d}}+1}{\kappa^-_{\bj_{-d}}+2+\rho\kappa^-_{\bj_{-d}}+\rho\frac{\kappa_{\bj_{-d}}^+-\kappa_{\bj_{-d}}^-}{\kappa_{\bj_{-d}}^++1}}}\\
			&\le 
			\rbr{\frac{1}{n}}^{\frac{\kappa^-_{\bj_{-d}}+1}{\kappa^-_{\bj_{-d}}+2+\rho\kappa^-_{\bj_{-d}}+\rho{C_K(\sqrt{d} /M)^\alpha}}}\\
			&= 
			\rbr{\frac{1}{n}}^{\frac{\kappa^-_{\bj_{-d}}+1}{\kappa^-_{\bj_{-d}}+2+\rho\kappa^-_{\bj_{-d}}} + \frac{(\kappa^-_{\bj_{-d}}+1)\rho C_K(\sqrt{d} /M)^\alpha}{(\kappa^-_{\bj_{-d}}+2+\rho\kappa^-_{\bj_{-d}}+\rho{C_K(\sqrt{d} /M)^\alpha})(\kappa^-_{\bj_{-d}}+2+\rho\kappa^-_{\bj_{-d}})}}\\
			&= 
			O\rbr{\frac{1}{n}}^{\frac{\kappa^-_{\bj_{-d}}+1}{\kappa^-_{\bj_{-d}}+2+\rho\kappa^-_{\bj_{-d}}}}.
		\end{align*}
		The last equality follows from the fact that $M = \log n$ and $n^{-1/\log n} = O(1)$. 
		Since $\kappa$ is defined as the overall minimum, under $E_\xi$, we have
		\begin{align*}
			\sup_{f^*\in \cF(d^*,\beta^{*}) }\EE(R(\hat{f}_n)-R({f^*})) 
			&\lesssim  \sum_{{\bj_{-d}}\in J_M} \rbr{\frac{1}{n}}^{\frac{\kappa^-_j+1 }{\kappa^-_j+2 +\rho\kappa^-_j}}\cdot \log^4(n)\\
			&=O\rbr{n^{-\frac{(\kappa+1)\beta^{*}}{(\kappa+2)\beta^{*}+\kappa d^*}}(\log n)^{d+3}}.
		\end{align*}
	\end{proof}

	\begin{remark}[$\xi$, $\epsilon$ and $B_n$]
		$E_\xi$ is defined for technical purposes to join local regions $D_{\bj_{-d}}$ so that \eqref{eqn:global} can hold with high probability. 
		Note that in order to satisfy $nM^{d-1}\xi^{d-1}\to 0$, $\xi$ can be chosen to be arbitrarily small, e.g., much smaller than the required approximation error $\epsilon$. However, this is at the expense of larger DNN derivatives (piecewise linear part in the helper functions in the construction process) within the $D_\xi$ band. 
		See Figure \ref{fig:ab} for illustration. 
		If $\xi= o(\epsilon)$, the Lipschitz constants of the constructed $c$ and $h$ in the Section \ref{sec:p13} will tend to infinity, which means the parameters of $c$ and $h$ will diverge, i.e., the corresponding $B_n$ of $\cF_n$ will not be bounded. 
		Although the ReLU network family considered in \cite{schmidt2020nonparametric} has constrained $B_n=O(1)$, an unbounded $B_n$ is not a problem, and such constraint is not employed in \citep{lu2020deep}. 
	\end{remark}
	\begin{figure}
		\centering
		\includegraphics[width=0.45\textwidth]{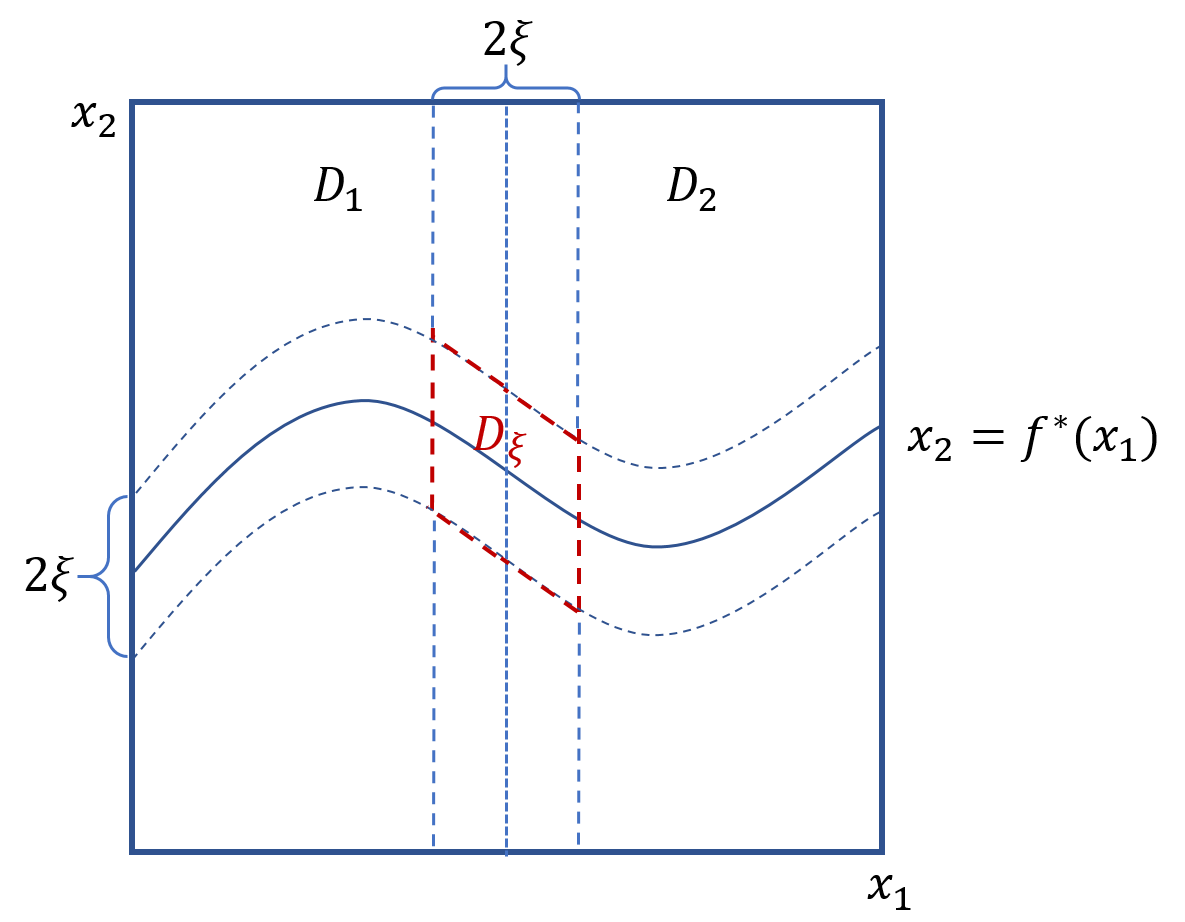}
		\caption{Illustration of region $D_\xi$ in $d=2, M=1$ case.}
		\label{fig:E2}
	\end{figure}
	In Theorem \ref{thm:global_estimator}, if we want $B_n=O(1)$, we can utilize the approximation results with such constraint, e.g., Lemma \ref{smooth} \citep{schmidt2020nonparametric}, and modify the region $D_\xi$ to be more efficient as shown in Figure \ref{fig:E2}, where we have
	$
	\PP(x\in D_\xi)\lesssim M\xi^2 d.
	$
	If we further assume $\rho<1$ and 
	$\kappa^->\frac{2}{1-\rho}$, we can ensure $\epsilon=O(\xi)$ by choosing $\xi\asymp n^{-\frac{\kappa^-}{\kappa^-+2+\rho\kappa^-}}$. In this case, $nM\xi^2 d\to 0$ and $\PP(E_\xi)\to 1$ as $n\to \infty$.

	\subsection{Lower bound - proof of Theorem \ref{thm:lower}}

	The lower bound result comes from estimation of sets in the discriminative analysis \citep{mammen1999smooth} where two independent samples $\cX^+=\{\bx^+_1,\cdots, \bx^+_n\}$ and $\bx^-=\{\bx^-_1,\cdots, \bx^-_m\}$ of $\RR^d$-valued i.i.d. observations with unknown densities $f$ or $g$ respectively (w.r.t. a $\sigma$-finite measure $Q$) are given. The goal is to predict whether a new sample $\bx$ is coming from $f$ or $g$ with a discrimination decision rule defined by a set $G\subset \RR^d$ that we attribute $\bx$ to $f$ if $\bx\in G$ and to $g$ otherwise. Let the Bayes risk to be
	$$R(G) = \frac{1}{2}\left(\int_{G^c}f(\bx)Q(d\bx)+\int_G g(\bx)Q(d\bx)\right)$$
	Denote $G^* = \{\bx:f(\bx)\ge g(\bx)\}$ to be the Bayes risk minimizer and consider the distance between two sets $G_1, G_2$ to be
	$$d_{f, g}(G_1,G_2) =\int_{G_1\triangle G_2}|f(\bx) - g(\bx)|Q(d\bx)$$

	Let $\tilde{G}_{m,n}$ be an empirical rule based on observations.
	The excess risk can be expressed as $R(\tilde{G}_{m,n})-R(G^*)=\frac{1}{2}d_{f,g}(\tilde{G}_{m,n}, G^*)$. In the following, we establish how fast can the excess risk go to zero under the smooth boundary condition. 
	
	For positive constants $c_1, c_2, \eta_0, \kappa$ and for a $\sigma$-finite measure $Q$, consider densities $f, g$ on $\RR^d$ w.r.t. $Q$ and
	define class $\cF$ of paired densities to be 
	\begin{align*}
		\cF_{\cG} = \{&(f,g): Q\{\bx\in \cX: |f(\bx)-g(\bx)|\le \eta\}\le c_2\eta^\kappa \mbox{ for } 0\le\eta\le\eta_0,\\
		&\{\bx\in \cX: f(\bx)\ge g(\bx)\}\in\cG, f(\bx), g(\bx)\le c_1 \mbox{ for } x\in \cX
		\}
	\end{align*} 
	Now let the base measure $Q$ be the Lebesgue measure $\QQ$ and recall $d_{\triangle}(G_1, G_2) = Q(G_1\triangle G_2)$.
	The following lemma establishes the connection between $d_{\triangle}$ and $d_{f,g}$
	\begin{lemma}[Lemma 2 in \citep{mammen1999smooth}]
		\label{tri}
		There exists a constant $c(\kappa)$ depending on $\kappa$ such that for Lebesgue measurable subsets $G_1$ and $G_2$ of $\cX$ and for $(f,g)\in \cF_{\cG}$,
		\begin{align*}
			c(\kappa)d_{\triangle}^{(1+\kappa)/\kappa}(G_1, G_2)\le d_{f, g}(G_1, G_2)\le 2c_1d_{\triangle}(G_1, G_2)
		\end{align*}
	\end{lemma}

	\noindent\textit{Proof of Theorem \ref{thm:lower}}~
	
	\begin{proof}
		Without loss of generality, assume $n\le m$ so we mainly focus on $\cX^+$. Consider the subset of $\cF_{\cG_h}$ that contains all pairs $(f,g_0)$, where $g_0$ is fixed and $f$ belongs to a finite class of densities $\cF_1$ that will be defined later. Then,
		\begin{align*}
			\sup_{(f,g)\in \cF_{\cG_h}}\EE_{f, g}d_{\triangle}(\tilde{G}_{m,n}, G^*)&\ge\sup_{(f,g_0):f\in\cF_1} \EE_{f,g}d_{\triangle}(\tilde{G}_{m,n}, G^*)\\
			&\ge \EE_{g_0}\left[\frac{1}{|\cF_1|}\sum_{f\in\cF_1}\EE_f[d_\triangle(\tilde{G}_{m,n}, G^*)|y_1, \cdots, y_m]\right]
		\end{align*}
		where $\EE_f$ and $\EE_{g_0}$ denotes the expectations w.r.t. the distributions of $(x_1,\cdots, x_n)$ and $(y_1,\cdots,y_m)$ when the underlying densities are $f$ and $g_0$.
		
		Recall the compositional assumption (\ref{comp}) and let $$i^* \in \argmax_{i=0,1,\cdots, q} n^{-\frac{2\bar{\beta}_i}{2\bar{\beta}_i+t_i}}\quad \mbox{  and  } \quad \beta^* = \beta_{i^*}$$
		Further denote
		$B=\prod_{l=i^*+1}^q(\beta_l\wedge 1)$ and then $\beta^{*} = \beta^* B$.
		For simplicity, we give the proof for the case $d^*=t_{i^*}=1$, that is the effective dimension of the smooth boundaries is 1 instead of $d-1$. For this case, let $\phi\in \cC_1^{\beta^*}(\RR, 1)$ be a real-valued function supported on $[-1,1]$ with $\phi(t)\ge 0$ for all $t$, $\max \phi(t)=1$ and $\phi(0)=1$. For $x=(x_1, \cdots, x_d)\in [0,1]^d$, define
		\begin{align*}
			g_0(\bx)= &(1-\eta_0-b_1)\II_{\{0<x_2<\frac{1}{2}\}} + \II_{\{\frac{1}{2}\le x_2<\frac{1}{2}+(\tau M^{-\beta^*})^B\}} \\
			&+(1+\eta_0+b_2)\II_{\{\frac{1}{2}+(\tau M^{-\beta^*})^B\le x_2\le 1\}}
		\end{align*}
		where $M\ge 2$ is an integer to be specified later and $\tau\in(0,1)$ is a constant. $b_1 = (\tau M^{-\beta^*}/c_2)^{B/\kappa}$ and $b_2>0$ is chosen such that $g_0$ integrates to 1.
		For $j=1,2,\cdots, M$ and $t\in[0,1]$, let 
		$$\psi_j(t) = \tau M^{-\beta^*}\phi\left(M\left[t-\frac{j-1}{M}\right]\right)$$
		Note that $\psi_j$ is only supported on $[\frac{j-1}{M}, \frac{j}{M}]$. 
		For vectors $\omega = (\omega_1,\cdots,\omega_M)$ with elements $w_j\in\{0,1\}$, define
		\begin{equation}
			b_\omega(t) = \sum_{j=1}^M \omega_j\psi_j(t)
		\end{equation}
		Now we construct functions in $\cH(l, \mathbf{d}, \mathbf{t}, \bbeta, R)$. For $i< i^*$, let $g_i(\bx):=(x_1, \cdots, x_{d_i})^\intercal$. For $i=i^*$ define $g_{i^*,\omega}(\bx)=(b_\omega(x_1), 0,\cdots, 0)^\intercal$. For $i>i^*$, set $g_i(\bx):=(x_1^{\beta_i\wedge 1},0,\cdots, 0)^\intercal$. 
		\begin{align*}
			\tilde{b}_{\omega}(\bx) &= g_q\circ \cdots\circ g_{i^*+1}\circ g_{i^*, \omega}\circ g_{i^*-1}\circ\cdots\circ g_0(\bx)=b_\omega(x_1)^B
		\end{align*}
		Notice that $\tilde{b}_\omega(\bx)\le (\tau M^{-\beta^*})^B$. Let $\Omega=\{0,1\}^M$. Define
		\begin{align*}
			f_\omega(\bx)& = 1+\left[\frac{\frac{1}{2}+ (\tau M^{-\beta^*})^B-x_2}{c_2}\right]^{1/\kappa}\II_{\{\frac{1}{2}\le x_2\le\frac{1}{2}+\tilde{b}_\omega(\bx)\}}-b_3(\omega)\II_{\{\frac{1}{2}+\tilde{b}_\omega(\bx)< x_2\le 1\}}
		\end{align*}
		where $b_3(\omega)>0$ is chosen such that $f_\omega(x)$ integrates to 1. Note that both $g_0(\bx)$ and $f_\omega(\bx)$ are $d$-dimensional densities even though they seem to only depend on $x_1$ and $x_2$. Other entries follow independent uniform distribution on $[0,1]$ and don't show on the density formulas.
		
		Set $\cF_1 = \{f_\omega: \omega\in \Omega\}$ and we will show that $(f_\omega, g_0)\in\cF_{\cG_h}$ for all $\omega\in\Omega$. To this end, we need to verify that 
		\begin{itemize}
			\item[(a)] 
			$f_\omega(\bx)\le c_1$ for $x\in K$; 
			\item[(b)]
			$\{\bx\in \cX: f_\omega(\bx)\ge g_0(\bx)\}\in \cG_h$;
			\item[(c)]
			$Q\{\bx\in \cX: |f_\omega(\bx)-g_0(\bx)|\le\eta\}\le c_2\eta^{\kappa}$ for all $0<\eta<\eta_0$.
		\end{itemize} 
		For (a), since $f_\omega$ integrates to 1, 
		\begin{align*}
			b_3(\omega)&\le \max_{\{\frac{1}{2}\le x_2\le\frac{1}{2}+\tilde{b}_\omega(\bx)\}} \left[\frac{\frac{1}{2}+ (\tau M^{-\beta^*})^B-x_2}{c_2}\right]^{1/\kappa}\\
			& \le \left[\frac{2\tau^B M^{-\beta^{*}}}{c_2}\right]^{1/\kappa}=\cO(M^{-\beta^{*}/\kappa})
		\end{align*}
		
		Thus, $f_\omega(\bx)\le c_1$ for $c_1$ and $M$ large enough. 
		
		(b) is satisfied since
		\begin{align*}
			\{\bx: f_\omega(\bx)\ge g_0(\bx)\} = \{\bx: 0\le x_2\le \frac{1}{2}+\tilde{b}_\omega(x_1)\}    
		\end{align*}
		and by construction, $\tilde{b}_\omega(\bx)\in \cH(l, \mathbf{d}, \mathbf{t}, \bbeta, R)$ for $\tau$ small enough.
		
		(c) follows that
		\begin{align*}
			&Q\{\bx\in \cX: |f_\omega(\bx)-g_0(\bx)|\le\eta\}\\
			\le& Q\{\bx\in \cX: \frac{1}{2}\le x_2\le \frac{1}{2}+(\tau M^{-\beta^*})^B, \left[\frac{1/2+(\tau M^{-\beta^*})^B-x_2}{c_2}\right]^{1/\kappa}\le\eta\}\\
			\le &Q \{\bx\in \cX: \frac{1}{2}+ (\tau M^{-\beta^*})^B-c_2\eta^\kappa \le x_2 \le \frac{1}{2}+(\tau M^{-\beta^*})^B\}\\
			\le & c_2\eta^\kappa
		\end{align*}
		After verifying $(f_\omega, g_0)\in\cF_{\cG_h}$ for all $\omega\in\Omega$, we now establish how fast can 
		$$
		S:=\frac{1}{|\cF_1|}\sum_{f\in\cF_1}\EE_f[d_\triangle(\tilde{G}_{m,n}, G^*)|y_1, \cdots, y_m]
		$$ 
		go to zero. To this end, we use the Assouad's lemma stated in \citep{korostelev2012minimax}
		which is adapted to the estimation of sets. 
		
		For $j=1, \cdots, M$ and for a vector $\omega=(\omega_1,\cdots,\omega_M)$, we write 
		\begin{align*}
			& \omega_{j0} = (\omega_1,\cdots, \omega_{j-1}, 0 , \omega_{j+1},\cdots, \omega_M)\\
			&   \omega_{j1} = (\omega_1,\cdots, \omega_{j-1}, 1 , \omega_{j+1},\cdots, \omega_M)
		\end{align*}
		For $i=0$ and $i=1$, let $P_{ji}$ be the probability measure corresponding to the distribution of $x_1,\cdots, x_n$ when the underlying density is $f_{\omega_{ji}}$. Denote the expectation w.r.t. $P_{ji}$ as $\EE_{ji}$. Let 
		\begin{align*}
			\cD_j &= \{\bx\in \cX: \frac{1}{2}+\tilde{b}_{\omega_{j0}}(\bx)<x_2\le\frac{1}{2}+\tilde{b}_{\omega_{j1}}(\bx)\}\\
			&=\{\bx\in \cX: b_{\omega_{j0}}(x_1)<\left(x_2-\frac{1}{2}\right)^{1/B}\le b_{\omega_{j1}}(x_1)\}\\
			& = \{\bx\in \cX:b_{\omega_{j0}}(x_1)<\left(x_2-\frac{1}{2}\right)^{1/B}\le b_{\omega_{j0}}(x_1)+ \psi_j(x_1)\}
		\end{align*}
		Then 
		\begin{align*}
			S&\ge\frac{1}{2}\sum_{j=1}^M Q\{\cD_j\}\int\min\{dP_{j1},dP_{j0}\}\\
			&\ge \frac{1}{2}\sum_{j=1}^M\int_0^1 \psi_j(x_1)^Bdx_1\int\min\{dP_{j1},dP_{j0}\}\\
			&\ge \frac{1}{2}\sum_{j=1}^M \tau^B M^{-\beta^{*}}\int\phi(Mt)^Bdt\int\min\{dP_{j1},dP_{j0}\}\\
			&\ge \frac{1}{4}\sum_{j=1}^M \tau^B M^{-\beta^{*}}\int\phi(Mt)^Bdt\left[1-H^2(P_{10},P_{11})/2\right]^n
		\end{align*}
		where $H(\cdot,\cdot)$ denotes the Hellinger distance. 
		
		\begin{align*}
			H^2(P_{10},P_{11}) =& \int \left[\sqrt{f_{\omega_{10}}(\bx)}-\sqrt{f_{\omega_{11}}(\bx)}\right]^2d\bx\\
			\le&
			\int_{0}^1\Bigg\{\int_{\frac{1}{2}}^{\frac{1}{2}+\psi_1(x_1)^B}\left[1-\sqrt{1+\left(\frac{\frac{1}{2}+\tau^B M^{-\beta^{*}}-x_2}{c_2}\right)^{1/\kappa}}\right]^2dx_2 \\
			&+
			\int_{\frac{1}{2}}^1\left[\sqrt{1-b_3(\omega_{10})}-\sqrt{1-b_3(\omega_{11})}\right]^2dx_2\Bigg\}dx_1\\
			\le&
			\int_{0}^1\int_{(\tau M^{-\beta^{*}})^B-\psi_1(x_1)^B}^{(\tau M^{-\beta^{*}})^B}\left[1-\sqrt{1+\left(\frac{v}{c_2}\right)^{1/\kappa}}\right]^2dvdx_1 \\
			&+
			|b_3(\omega_{10})-b_3(\omega_{11})|^2
		\end{align*}
		For the first term,
		\begin{align*}
			&\int_{0}^1\int_{\tau^B M^{-\beta^{*}}-\psi_1(x_1)^B}^{\tau^B M^{-\beta^{*}}}\left[1-\sqrt{1+\left(\frac{v}{c_2}\right)^{1/\kappa}}\right]^2dvdx_1\\
			&\le \int_{0}^1\int_{\tau^B M^{-\beta^{*}}-\psi_1(x_1)^B}^{\tau^B M^{-\beta^{*}}}\left(\frac{v}{c_2}\right)^{2/\kappa}dvdx_1 \\
			&\le\frac{\kappa c_2^{-2/\kappa}}{\kappa+2}\int_0^1\left(\tau^B M^{-\beta^{*}}\right)^{1+2/\kappa}- \left(\tau^B M^{-\beta^{*}}-\psi_1(x_1)^B\right)^{1+2/\kappa}dx_1\\
			&\le\frac{\kappa c_2^{-2/\kappa}}{\kappa+2}\left(\tau^B M^{-\beta^{*}}\right)^{1+2/\kappa}\int\left(1-(1-\phi(Mt)^B)^{1+2/\kappa}\right)dt\\
			&=\cO\left(M^{-\beta^{*}(1+2/\kappa)-1}\right)
		\end{align*}
		On the other hand, 
		\begin{align*}
			\int_0^1\int_{1/2}^{1/2 + b_\omega(x_1)^B}\left[\frac{\frac{1}{2}+ \tau^B M^{-\beta^{*}}-x_2}{c_2}\right]^{1/\kappa}dx_2dx_1 = b_3(\omega)\left[\frac{1}{2}-b_\omega(x_1)^B\right]
		\end{align*}
		yields
		\begin{align*}
			b_3(\omega_{11})&=\frac{1}{\frac{1}{2}-b_{\omega_{11}}(x_1)^B}\int_0^1\int_{1/2}^{1/2 + b_{\omega_{11}}(x_1)^B}\left[\frac{\frac{1}{2}+ \tau^B M^{-\beta^{*}}-x_2}{c_2}\right]^{1/\kappa}dx_2dx_1\\
			&\le \frac{M c_2^{-1/\kappa}}{\frac{1}{2}-\tau^B M^{-\beta^{*}}}\int_0^1\int_{\tau^B M^{-\beta^{*}}(1-\phi(Mx_1))}^{\tau^B M^{-\beta^{*}}}
			u^{1/\kappa}dudx_1\\
			&=\frac{Mc_2^{-1/\kappa}\tau^B}{(\frac{1}{2}-\tau^B M^{-\beta^{*}})(1+1/\kappa)}
			M^{-\beta^{*}(1+1/\kappa)}\int(1-(1-\phi(Mt)^B)^{1+1/\kappa})dt\\
			&\le\frac{c_2^{-1/\kappa}\tau^B}{(\frac{1}{2}-\tau^B M^{-\beta^{*}})(1+1/\kappa)}
			M^{-\beta^{*}(1+1/\kappa)}\\
			&=\cO(M^{-\beta^{*}(1+1/\kappa)}) 
		\end{align*}
		Hence $|b_3(\omega_{11})-b_3(\omega_{10})| = \cO(M^{-\beta^{*}(1+1/\kappa)-1}) $ and we have
		\begin{align*}
			H^2(P_{10}, P_{11})&=\cO\left(M^{-\beta^{*}(1+2/\kappa)-1} \vee M^{-\beta^{*}(2+2/\kappa)-2}\right) \\
			&=\cO\left(M^{-\beta^{*}(1+2/\kappa)-1}\right)
		\end{align*}
		Now choose $M$ as the smallest integer that is larger or equal to 
		$$n^{\frac{\kappa}{(2+\kappa)\beta^{*}+\kappa}}$$
		Then we have $ H^2(P_{10}, P_{11})\le C^* n^{-1}\left(1+ o(1)\right)$ for some constant $C^*$ depending only on $\kappa, c_2,\tau, \phi$ and 
		$$\int\min\{dP_{j1},dP_{j0}\}\ge \frac{1}{2}\left[1-\frac{C^*}{2}n^ {-1}(1+o(1))\right]^n\ge C_1^*$$
		for $n$ large enough and $C_1^*$ is another constant. Thus for $n$ large enough,
		\begin{align*}
			S\ge \frac{1}{2}C_1^*\tau^B M^{-\beta^ {*}}\int \phi(t)dt\ge C_2^*n^{-\frac{\kappa\beta^{*}}{(2+\kappa)\beta^{*}+\kappa}}
		\end{align*}
		The constant $C_2^*$ only depends on $\kappa, c_2,\tau$ and $ \phi$.
		
		Combining all the results so far yields that 
		\begin{align*}
			\liminf_{n\to\infty}\inf_{\tilde{G}_{m,n}}\sup_{(f,g)\in \cF_{\cG_{h}}}
			(n\wedge m)^{\frac{\beta^{*}\kappa}{\beta^{*}(\kappa+2)+d^*\kappa }}\EE_{f,g}[d_{\triangle}(\tilde{G}_{m,n},G^*)]>0
		\end{align*}
		holds when $d^*=1$. Using Lemma \ref{tri}, we have
		\begin{align*}
			\liminf_{n\to\infty}\inf_{\tilde{G}_{m,n}}\sup_{(f,g)\in \cF_{\cG_{h}}}
			(n\wedge m)^{\frac{\beta^{*}(\kappa+1)}{\beta^{*}(\kappa+2)+d^*\kappa }}\EE_{f,g}[d_{f, g}(\tilde{G}_{m,n},G^*)]>0
		\end{align*}
	\end{proof}
	Note that under the composition smoothness assumption on the decision boundaries, existing upper bounds are suboptimal. The sub optimality comes from the $\kappa$ term in the denominator of the rates and the $\log n$ term. Instead of $\kappa$ as in the lower bound, we have $\kappa+1$ in the upper bound. 
	
	\subsection{Numerical experiment details}
	\label{sec:exp}
	All experiments are conducted using PyTorch. The ReLU neural networks are coded with the \texttt{torch.nn} package, with the default initialization. Sample size is chosen to be 1000 for the experiments in Figure \ref{fig:acc}. 
	The neural network training is done by stochastic gradient descent (initial learning rate=0.1, momentum=0.9, weight decay=0.001). The batch size is chosen to be 100. The total iteration number is 10000, with learning rate decayed by 1/10 every 2000 steps. 
	To make a fair comparison, we fix the random seed for the data generating process. The randomness comes from network initialization and batch selection. Test accuracy is evaluated by sampling one million test data points.  
	When producing Figure \ref{fig:rate}, we choose $k=10$ and $n=200, 800, 3200, 12800$. Each point in the figure is an average from 10 replications.

	\vskip 0.2in
	\bibliography{references}

\end{document}